\documentclass[a4paper]{scrartcl}
\usepackage[utf8]{inputenc}
\usepackage{amsmath,amssymb,amsthm,amsfonts}
\usepackage{mathtools,thmtools}
\usepackage[british]{babel}
\usepackage{csquotes}
\usepackage{hyperref}
\usepackage[nameinlink,noabbrev,capitalise]{cleveref}
\usepackage{enumerate}
\usepackage{xcolor}
\usepackage{comment}
\usepackage{romanbar}
\usepackage{bm}

\RequirePackage{tikz}
    \usetikzlibrary{positioning}
    \usetikzlibrary{decorations.pathreplacing}

\usepackage[maxbibnames=5,sorting=nyt,sortcites,giveninits=true]{biblatex}
\DeclareNameAlias{sortname}{family-given}
\usepackage[T2A,T1]{fontenc}
\DeclareSymbolFont{cyrillic}{T2A}{cmr}{m}{n} 
\DeclareMathSymbol{\rusp}{\mathalpha}{cyrillic}{239}
\DeclareMathSymbol{\rusz}{\mathalpha}{cyrillic}{214}
\DeclareMathSymbol{\rusg}{\mathalpha}{cyrillic}{227}
\DeclareMathSymbol{\rush}{\mathalpha}{cyrillic}{237}
\DeclareMathSymbol{\rusd}{\mathalpha}{cyrillic}{196}

\newcommand{\R}     {\mathbb{R}}
\newcommand{\N}     {\mathcal{N}}

\renewcommand{\P}   {\mathbb{P}}
\newcommand{\E}     {\mathbb{E}}

\newcommand{\eps}   {\varepsilon}
\let\S\relax
\newcommand{\S}     {\mathcal{S}}
\newcommand{\calP}     {\mathcal{P}}

\newcommand{\TV}    {\textup{TV}}
\newcommand{\W}{\mathcal{W}}
\newcommand{\tmix}{t_{\operatorname{mix}}}
\newcommand{\ind}{\mathbf{1}}
\newcommand{\diam}{\operatorname{diam}}

\newcommand{\Xf}{\vec{X}}
\newcommand{\Xb}{\cev{X}}
\newcommand{\Pf}{\vec{P}}
\newcommand{\Pb}{\cev{P}}
\newcommand{\Yf}{\vec{Y}}
\newcommand{\Yb}{\cev{Y}}
\newcommand{\Qf}{\vec{Q}}
\newcommand{\Qb}{\cev{Q}}
\newcommand{\comp}{\mathsf{c}}

\newcommand{\cev}[1]{\reflectbox{\ensuremath{\skew{-8}\vec{\reflectbox{\ensuremath{#1}}}}}}

\DeclarePairedDelimiterX{\norm}[1]{\lVert}{\rVert}{#1}

\DeclareMathOperator{\Gap}  {Gap}

\DeclareMathOperator{\Var}  {Var}

\DeclareMathOperator{\Unif} {Unif}

\DeclareMathOperator{\MH}  {MH}

\DeclareMathOperator{\Law}  {Law}

\theoremstyle{plain}
\newtheorem{theorem}{Theorem}
\newtheorem{lemma}[theorem]{Lemma}
\newtheorem{proposition}[theorem]{Proposition}
\newtheorem{corollary}[theorem]{Corollary}

\theoremstyle{definition}
\newtheorem{definition}{Definition}
\newtheorem{example}{Example}
\newtheorem{remark}{Remark}

\title{Local Geometric Mixing via Dobrushin Contraction with Applications to Diffusion Path Monte Carlo and the Proximal Sampler}
\author{%
  \Large Stefan Oberd\"orster%
  \thanks{%
    E-mail:\ \href{mailto:oberdoerster@uni-bonn.de}
                 {\texttt{oberdoerster@uni-bonn.de}},
    ORCID:\ \href{https://orcid.org/0009-0004-0734-0101}
                {0009-0004-0734-0101}.%
  }\\[0.6em]
  \large Institute for Applied Mathematics, University of Bonn
}
\date{\normalsize\today}

\begin{document}

\maketitle

\begin{abstract}
\noindent
Local geometric mixing localizes geometric mixing by requiring geometric convergence to equilibrium in total variation only over finitely many transitions.
It accommodates local convergence rates and captures rapid local equilibration, even when global mixing is much slower.
We establish and discuss local geometric mixing bounds through Dobrushin contraction.
We then apply this approach to Diffusion Path Monte Carlo, a recently proposed Markov chain Monte Carlo method, aimed at leveraging advances in score-based modeling, whose ideal transitions coincide with those of the Proximal Sampler.
Our analysis covers both the ideal method and its implementable Metropolis-adjusted counterpart, providing mixing guarantees under minimal assumptions.
For the ideal method, these guarantees complement recent spectral gap estimates, which we develop into mixing time bounds.
\end{abstract}

\section{Introduction}

Let $(\S,\mathfrak B)$ be a Polish state space endowed with its Borel $\sigma$-algebra, and denote the set of probability measures on $(\S,\mathfrak B)$ by $\calP(\S)$.
Considering a Markov transition kernel $\pi$ on $(\S,\mathfrak B)$ with invariant distribution $\mu\in\calP(\S)$, a central aspect of study is \emph{mixing}, that is, convergence to equilibrium
\[ \TV\bigl(\nu\pi^n,\mu\bigr)\ \longrightarrow\ 0\quad\text{as $n\to\infty$} \]
from a given initial distribution $\nu\in\calP(\S)$ in the total variation distance, defined by
\[ \mathrm{TV}(\eta,\tilde\eta)\ :=\ \sup\nolimits_{B\in\mathfrak{B}}\,\bigl|\eta(B)-\tilde\eta(B)\bigr|\quad\text{for any $\eta,\tilde\eta\in\calP(\S)$.} \]

A classical approach to analyze this convergence, going back to the work of Dobrushin \cite{Dobrushin56}, is to consider contraction of the Markov kernel in total variation.
If $\pi$ is a global contraction, that is, there exists $\rho>0$ such that
\begin{equation}\label{eq:globDob}
    \TV\bigl(\pi(x,\cdot),\pi(\tilde x,\cdot)\bigr)\ \leq\ 1-\rho\quad\text{for all $x,\tilde x\in\S$,}
\end{equation}
then
\begin{equation}\label{eq:unifgeomix}
    \TV\bigl(\nu\pi^n,\mu\bigr)\ \leq\ (1-\rho)^n\quad\text{for all $\nu\in\calP(\S)$ and $n\geq0$.}
\end{equation}
Condition \eqref{eq:globDob} is referred to as \emph{global Dobrushin contraction} and its consequence \eqref{eq:unifgeomix} as \emph{uniform geometric mixing}.
Note that both imply existence and uniqueness of the invariant distribution $\mu$.
Commonly, an additional multiplicative constant $C\geq1$ is allowed on the right hand side, which, to describe non-uniform geometric mixing, may depend on $\nu$.

Closely related to global Dobrushin contraction are the generally stronger \emph{minorization conditions}, dating back to the work of Doeblin \cite{DOEBLIN37,Doeblin40}, which assert the existence of $\eta\in\calP(\S)$ and $\rho>0$ such that
\begin{equation}\label{eq:globmin}
    \pi(x,\cdot)\ \geq\ \rho\,\eta\quad\text{for all $x\in\S$.}
\end{equation}

In many instances, (uniform) geometric mixing is too strong of a conclusion to fully capture the convergence to equilibrium of a Markov chain.
For illustration, consider the following example.

\begin{example}[Seldom-trapped Chain]\label{ex:seldomtrapped}
For $0<\epsilon\ll1$, on $\S=\mathbb S^1\cup\{0\}$, consider the Markov kernel
\[ \pi(x,\cdot)\ =\
    \begin{cases}
    \displaystyle\frac{1-\epsilon^2}{2}\,\operatorname{Unif}(\mathbb S^1)\ +\ \frac{\epsilon^2}{2}\,\delta_0\ +\ \frac12\,\delta_x & \text{if $x\in\mathbb S^1$,}\\[6pt]
    \displaystyle\frac{\epsilon(1-\epsilon)}{2}\,\operatorname{Unif}(\mathbb S^1)\ +\ \left(1-\frac{\epsilon(1-\epsilon)}{2}\right)\,\delta_0 & \text{if $x=0$}
    \end{cases} \]
with invariant distribution $\mu=(1-\epsilon)\,\Unif(\mathbb S^1)+\epsilon\,\delta_0$.

On the circle, both with probability close to $1/2$, the chain remains in its previous state or freely moves to an independent draw from the uniform distribution, whereas, in the origin, the chain is trapped and only escapes to the circle with a small probability.
From start on the circle, this induces mixing on two distinct time scales: initially with fast rate $1/2$ until the chain equilibrates on the circle, followed by slow exploration of the origin with rate $\epsilon/2$.
From start in the origin, the small escape probability restricts mixing to the slow rate $\epsilon/2$.
Indeed, explicit calculations show that, for all $n\geq0$,
\[ \TV\bigl(\pi^n(x,\cdot),\mu\bigr)\ =\
    \begin{cases}
    \displaystyle\max\left(2^{-n},\epsilon\left(1-\frac{\epsilon}{2}\right)^n\right) & \text{if $x\in\mathbb S^1$,}\\[6pt]
    \displaystyle(1-\epsilon)\left(1-\frac{\epsilon}{2}\right)^n & \text{if $x=0$.}
    \end{cases} \]

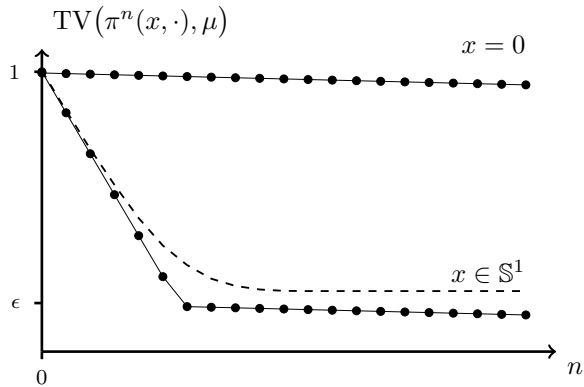
\begin{figure}[t]
\centering
\begin{tikzpicture}[scale=2,
  axis/.style={line width=1pt},
  plot label/.style={font=\small}]

\def\chainEps{0.02}
\edef\logEps{\fpeval{ln(\chainEps)}}
\edef\logOrigin{\fpeval{ln(1-\chainEps)}}
\edef\logRate{\fpeval{ln(1-\chainEps/2)}}
\edef\logMin{\fpeval{
  ln(10)*floor((min(\logEps,\logOrigin)+20*\logRate)/ln(10))
}}
\def\Y#1{1.8*(1-(#1)/\logMin)}

\draw[axis,->] (0,-0.05) -- (3.4,-0.05)
  node[plot label,below right] {$n$};
\draw[axis,->] (0,-0.05) -- (0,1.95)
  node[plot label,above right]
  {$\mathrm{TV}\bigl(\pi^n(x,\cdot),\mu\bigr)$};
\draw[axis] (0,-0.10) -- (0,-0.05)
  node[below,yshift=-3pt] {$\scriptstyle 0$};
\foreach \y/\label in {1.8/1,{\Y{\logEps}}/\epsilon} {
  \draw[axis] (-0.05,{\y}) -- (0,{\y})
    node[left,xshift=-4pt] {$\scriptstyle\label$};
}

\foreach \curve/\label in {
  {max(-\n*ln(2),\logEps+\n*\logRate)}/{x\in\mathbb S^1},
  {\logOrigin+\n*\logRate}/{x=0}} {
  \draw[line width=0.35pt]
    plot[variable=\n,domain=0:20,samples=21,
         mark=*,mark size=0.75pt,mark options={draw=none}]
    ({0.16*\n},{\Y{\curve}})
    node[plot label,anchor=south east,inner sep=0pt,yshift=12pt]
    {$\label$};
}
\edef\boundSlope{\fpeval{\chainEps^2}}
\edef\boundK{\fpeval{
  max(0,floor(-ln(\boundSlope)/ln(2)))
}}

\draw[black,dashed,line width=0.7pt]
plot[variable=\n,domain=0:20,samples=21]
({0.16*\n},
 {\Y{ln(2^(-min(\n,\boundK))
       +\boundSlope*(min(\n,\boundK)-1)
       +\chainEps)}});
\end{tikzpicture}
\caption{\textit{Mixing of the seldom-trapped chain in \Cref{ex:seldomtrapped}, shown on a logarithmic vertical scale.
From start on the circle, the chain locally equilibrates rapidly.
The subsequent exploration of the origin, as well as escape to the circle from start in the origin, proceed at a far slower rate.
The dashed line depicts the local geometric mixing bound due to \Cref{cor:mixlocDob}, see \Cref{ex:seldomtrappedbound}.}}
\label{fig:seldomtrapped}
\end{figure}

As the rate $\rho$ prescribed for (uniform) geometric mixing, see \eqref{eq:unifgeomix}, is uniform in the initial distribution $\nu$, and the number of transitions $n$, this notion cannot capture the fast local equilibration from start on the circle, and is instead limited to the slow global convergence rate $\epsilon/2$.
Correspondingly, the global Dobrushin contraction, determined by the slow contraction between circle and origin, reads
\[ \sup_{x,\tilde x\in\S}\TV\bigl(\pi(x,\cdot),\pi(\tilde x,\cdot)\bigr)\ =\ 1-\frac\epsilon2. \]
\end{example}

In order to capture faster-than-global local equilibration, as in our example, it seems natural to relax the Dobrushin contraction to hold locally instead of globally, that is, for some subset $D\subseteq\S$ of state space, to prescribe the existence of $\rho>0$ such that
\begin{equation}\label{eq:locDob}
    \TV\bigl(\pi(x,\cdot),\pi(\tilde x,\cdot)\bigr)\ \leq\ 1-\rho\quad\text{for all $x,\tilde x\in D$.}
\end{equation}
In our example, local contraction on the circle improves the slow global rate $\epsilon/2$ to $1/2$, the rate of local equilibration on the circle.
However, such \emph{local Dobrushin contraction} is, on its own, not sufficient for mixing to a given invariant distribution.
For instance the seldom-trapped chain with $\epsilon=0$, for which local Dobrushin contraction on the circle holds, leaves any convex combination of $\Unif(\mathbb S^1)$ and $\delta_0$ invariant while only converging to a single one determined by the initial distribution.

A well-established approach to conclude geometric mixing from local Dobrushin contraction is \emph{Harris' ergodic theorem}.
It additionally assumes a global Lyapunov drift condition towards the domain $D$ in which contraction holds, see \eqref{eq:locDob}, ensuring the chain to return to $D$ sufficiently regularly.
Specifically, for $D$ being a suitable sublevel set of $V:\S\to[0,\infty)$, with $\gamma\in(0,1)$ and $K\geq0$, one considers drift conditions of the form
\begin{equation}\label{eq:drift}
    \pi V(x)\ \leq\ \gamma\,V(x)\ +\ K\quad\text{for all $x\in\S$.}
\end{equation}
Many instances of Harris' theorem appearing in the literature are formulated in terms of local minorization conditions, similar to \eqref{eq:globmin}, instead of Dobrushin contraction.

Harris' ergodic theorem is rooted in \cite{Harris56}, showing how recurrence ensures existence of an invariant measure, unique up to normalization.
Subsequent work used regeneration to analyze repeated returns \cite{Athreya78,Nummelin78} and established geometric convergence under suitable drift and minorization conditions \cite{Popov77,Nummelin82,Chan89,Meyn92}.
Later developments made convergence bounds explicit \cite{MeTw1994,rosenthal1995minorization,Roberts01,Douc04}.
Hairer and Mattingly gave a particularly direct proof, showing how drift and minorization together produce contraction in a weighted total variation distance \cite{Hairer11b}.
See the monographs \cite{Meyn09,Douc18} for reference.

Crucially for capturing local equilibration as in our example, Harris' ergodic theorem asserts geometric mixing, i.e., \eqref{eq:unifgeomix} with initialization dependent multiplicative constant $C(\nu)\geq1$ on the right hand side.
In particular, the rate $\rho$ is uniform in the initial distribution $\nu$ and number of transitions $n$.
This immediately rules out the ability of Harris' theorem, as presented thus far, to describe fast local equilibration from certain initial distributions in the presence of slow mixing from others.
These limitations are discussed in \cite{Roberts96b,Qin21}.

To accurately describe local equilibration, instead of geometric mixing, we consider \emph{local geometric mixing}.
Therein, assuming $\mu$ to be invariant for $\pi$, for a given initial distribution $\nu\in\calP(\S)$, one aims at establishing $N\geq0$ and $\rho>0$ such that
\begin{equation}\label{eq:locgeomix}
    \TV\bigl(\nu\pi^n,\mu\bigr)\ \leq\ (1-\rho)^n\quad\text{for all $n\leq N$.}
\end{equation}
Since the total variation distance to equilibrium is non-increasing, this is equivalent to
\[ \TV\bigl(\nu\pi^n,\mu\bigr)\ \leq\ \max\bigl((1-\rho)^n,\delta\bigr)\quad\text{for all $n\geq0$} \]
with a suitable $\rho,\delta>0$.
This notion fully localizes geometric mixing by restricting the number of transitions for which the total variation distance to equilibrium decays geometrically with a given rate.
Importantly, during these transitions, the chain may effectively only visit a subset of state space, which permits a local rate $\rho(\nu,N)$ depending on the initialization and considered number of transitions.
In contrast, geometric mixing requires every realization of the chain to eventually explore the entire stationary mass, limiting the geometric decay, which is prescribed to hold for all $n\geq0$, to a global rate.
Note that this notion of local geometric mixing is strictly weaker than geometric mixing and, in particular, does not guarantee existence and uniqueness of an invariant measure.
One may again add a multiplicative constant $C\geq1$ on the right hand side.

Natural objects to quantify local mixing are mixing times
\[ \tmix(\eps,\nu)\ :=\ \inf\bigl\{n\geq0:\TV(\nu\pi^n,\mu)\leq\eps\bigr\} \]
to accuracy $\eps>0$ from start in $\nu\in\calP(\S)$.
To describe uniform mixing across a collection of initial distributions, one may take the supremum in the previous display.
For instance,
\[ \tmix(\eps,D)\ :=\ \sup\nolimits_{x\in D}\tmix(\eps,\delta_x)\quad\text{for $D\subseteq\S$.} \]

Analyses targeting local geometric mixing, often in form of bounds
\begin{equation}\label{eq:locgeomixlit}
    \TV\bigl(\nu\pi^n,\mu\bigr)\ \leq\ C(1-\rho)^n\ +\ \delta\quad\text{for all $n\geq0$}
\end{equation}
for suitable $C\geq1$ and positive $\rho$ and $\delta$, appeared both in the analytic and the probabilistic literature.
Early examples include the conductance bounds of Lov\'asz and Simonovits \cite{Lovasz93}, localizing the analysis via the use of $s$-conductance, as well as Rosenthal's coupling bounds \cite{rosenthal1995minorization}, separating insufficient visits to a domain in which a local minorization condition yields geometric mixing.
Later, local geometric mixing was studied in the context of the Metropolis-adjusted Langevin Algorithm (MALA).
Bou-Rabee and Hairer \cite{Nawaf13} apply Harris' theorem to a restricted variant of MALA.
Eberle \cite{Eberle14} establishes a Wasserstein bound similar to \eqref{eq:locgeomixlit}, separating exit from a domain in which Wasserstein contraction yields geometric convergence.
More recently, Rosenthal and Yang \cite{Yang23}, aiming at local geometric mixing, generalize Harris' theorem, localizing both drift and minorization conditions via controlling exit events.
Atchad\'e \cite{Atchade21} studies local spectral gap arguments.

In \Cref{sec:locDobmix}, we present a simple yet effective approach to local geometric mixing, combining local Dobrushin contraction with exit control.
Compared to Rosenthal and Yang's local Harris' theorem, it replaces local drift and minorization by local Dobrushin contraction and, if said contraction holds, serves as an alternative to their result.
Beyond contraction in one step, their combination of local drift and minorization rather corresponds to Dobrushin contraction over several steps, see \cite{Diss}.

In \Cref{sec:DPMC}, we introduce Diffusion Path Monte Carlo, a sampling method recently proposed in \cite{sifanDP,Hill26} whose ideal transitions follow first an Ornstein--Uhlenbeck process, noising the target distribution, and subsequently its time-reversal in law.
Leveraging advances in score-based modeling, it aims at adding sufficient noise to overcome potential barriers and thereby efficiently sample target distributions in which classical methods suffer metastability.
We identify its transition kernel with that of the Proximal Sampler \cite{Titsias18,Lee21,Sinho22}.

In \Cref{sec:DPMCmix}, after reviewing previous work on the convergence to equilibrium of the Proximal Sampler and ideal Diffusion Path Monte Carlo, we apply the results of \Cref{sec:locDobmix}.
Under assumptions of sufficient noising, we show (i) local mixing in one transition for any target and initial distributions, and (ii) geometric local mixing for targets with finite variance and arbitrary initialization.
This considerably relaxes assumptions on targets and initialization made in previous work.
Moreover, a complimentary regularization estimate turns recently established spectral gap estimates into mixing guarantees without requiring a warm start.

In \Cref{sec:MDPMC}, we present Metropolis-adjusted Diffusion Path Monte Carlo \cite{sifanDP,Hill26}, an implementable variant of the ideal method, and demonstrate its reversibility with respect to general target probability distributions based on Tierney's formulation \cite{Tierney98} of Metropolis--Hastings in general state spaces.

In the final \Cref{sec:MDPMCmix}, we transfer the Dobrushin contraction-based mixing analysis of ideal Diffusion Path Monte Carlo to its Metropolis-adjusted counterpart, providing quantitative local mixing guarantees whenever Metropolis acceptance probabilities are locally controlled in a domain, in which the dynamics are suitably stable.
To go full circle, we observe Metropolis-adjusted Diffusion Path Monte Carlo to display the characteristic behavior of the seldom-trapped chain, introduced above to motivate local geometric mixing.

\section{Local Geometric Mixing via Dobrushin Contraction}\label{sec:locDobmix}

As above, let $\pi$ be a Markov transition kernel on a Polish state space $\S$ with invariant distribution $\mu\in\calP(\S)$.
Denote $\bar\S=\S\times\S$ endowed with the product $\sigma$-algebra $\bar{\mathfrak B}=\mathfrak B\otimes \mathfrak B$, the associated projections by $\rusp_1,\rusp_2:\bar\S\to\S$, and the diagonal $\Delta=\{(x,x):x\in\S\}$.
A probability distribution $\bar\eta\in\calP(\bar\S)$ is a \emph{coupling} of $\eta,\tilde\eta\in\calP(\S)$ iff $\bar\eta\circ\rusp_1^{-1}=\eta$ and $\bar\eta\circ\rusp_2^{-1}=\tilde\eta$.
Similarly, a Markov transition kernel $\bar\pi$ on $\bar\S$ is a \emph{Markovian coupling} of $\pi$ iff $\bar\pi((x,\tilde x),\cdot)\circ\rusp_1^{-1}=\pi(x,\cdot)$ and $\bar\pi((x,\tilde x),\cdot)\circ\rusp_2^{-1}=\pi(\tilde x,\cdot)$ for all $(x,\tilde x)\in\bar\S$.

\begin{theorem}\label{thm:mixlocDob}
Let $\pi$ be a Markov transition kernel with invariant measure $\mu\in\calP(\S)$.
If there exists $\bar D\subseteq\bar\S$ and $\rho>0$ such that
\begin{equation}\label{eq:Dobctrbar}
    \TV\bigl(\pi(x,\cdot),\pi(\tilde x,\cdot)\bigr)\ \leq\ 1-\rho\quad\text{for all $(x,\tilde x)\in\bar D$},
\end{equation}
then, for all $n\geq0$, $\nu\in\calP(\S)$ and couplings $\bar\eta$ of $\nu$ and $\mu$, it holds
\begin{equation}\label{eq:mixlocDobbar}
    \TV\bigl(\nu\pi^n,\mu\bigr)\ \leq\ (1-\rho)^n\ +\ \mathbb P_{\bar\eta}\bigl(\bar T<n\bigr),
\end{equation}
where $\bar T:=\inf\{n\geq0:\bar X_n\notin\bar D\}$ denotes the first exit time from $\bar D$ of any Markov chain with transition kernel $\bar\pi$ being a Markovian coupling of $\pi$ that realizes \eqref{eq:Dobctrbar} in the sense that $\bar\pi((x,\tilde x),\Delta^\comp)\leq1-\rho$ for all $(x,\tilde x)\in\bar D$.
\end{theorem}

The second term on the right hand side of \eqref{eq:mixlocDobbar} accounts for the probability that the coupling chain leaves $\bar D$ before time $n$, inside of which Dobrushin contraction \eqref{eq:Dobctrbar} yields geometric mixing.
Note that a Markovian coupling realizing \eqref{eq:Dobctrbar} necessarily exists, see \cite{Ollivier09,VillaniOT}.

As the total variation to equilibrium is non-increasing, \eqref{eq:mixlocDobbar} immediately improves to
\begin{equation}\label{eq:infimp}
    \TV\bigl(\nu\pi^n,\mu\bigr)\ \leq\ \inf\nolimits_{k\leq n}\left((1-\rho)^k + \mathbb P_{\bar\eta}\bigl(\bar T<k\bigr)\right),
\end{equation}
preserving the smallest bound obtained up to transition $n$, while the original right hand side may grow as $n\to\infty$.
As the right hand side generally does not vanish in the limit, the result is specifically aimed at local geometric mixing, see \eqref{eq:locgeomix}.
We illustrate its ability to accurately capture local equilibration based on the seldom-trapped chain in \Cref{ex:seldomtrappedbound} below.

\medskip

Specializing to $\bar D=D\times D$ for some $D\subseteq\S$ immediately yields the following simplification.

\begin{corollary}\label{cor:mixlocDob}
Let $\pi$ be a Markov transition kernel with invariant measure $\mu\in\calP(\S)$.
It holds for all $n\geq0$, $\nu\in\calP(\S)$, and $D\subseteq\S$ that
\begin{equation}\label{eq:mixlocDob}
    \TV(\nu\pi^n,\mu)\ \leq\ \left(\sup_{x,\tilde x\in D}\TV\bigl(\pi(x,\cdot),\pi(\tilde x,\cdot)\bigr)\right)^n+\ \P_{\nu}(T<n)\ +\ \P_{\mu}(T<n),
\end{equation}
where $T:=\inf\{n\geq0:X_n\notin D\}$ denotes the first exit time from $D$ of the Markov chain with transition kernel $\pi$.
\end{corollary}

Domains $\bar D$ satisfying \eqref{eq:Dobctrbar} for some $\rho>0$, equivalently formulated in terms of the existence of $(x,\tilde x)$-dependent minorization measures, are referred to as \emph{coupling sets} in \cite{Douc04}.
Sets $D\subseteq\S$ for which contraction holds in $D\times D$ are called \emph{pseudo-small}, generalizing \emph{small} sets, for which the minorization measure must be uniform in $(x,\tilde x)\in D\times D$, see \cite{Roberts01}.
Allowing the broader coupling sets is advantageous because they can be substantially larger than products of pseudo-small sets while retaining the same contraction constant, see \cite{Douc04,Diss}.

An analogous bound for Wasserstein distances appears in \cite{Eberle14}.
The specialization to total variation and extension to general coupling sets are of independent interest, providing a more direct approach to mixing.
Moreover, among general Wasserstein distances, contraction in total variation can be particularly straightforward to establish for certain Markov chains, including those involving Metropolis adjustment, see \cite{Diss,BouRabeeOberdoerster2024}.

\begin{proof}[Proof of \Cref{thm:mixlocDob}]
Let $\bar\pi$ be a Markovian coupling of $\pi$ realizing \eqref{eq:Dobctrbar}, and let $\bar\eta$ be a coupling of $\nu$ and $\mu$.
Let $(\bar X_n)_{n\geq0}$ have transition kernel $\bar\pi$ and denote its coupling time by
\[ \tau\ :=\ \inf\bigl\{n\geq0:\bar X_n\in\Delta\bigr\}. \]
By the coupling lemma, see \cite{Levin09}, we have
\[ \TV\bigl(\nu\pi^n,\mu\bigr)\ \leq\ \P_{\bar\eta}(\tau>n). \]
Localizing to $\bar D$ via isolating the exit probability yields
\[ \P_{\bar\eta}(\tau>n)\ \leq\ \P_{\bar\eta}\bigl(\tau>n,\bar T\geq n\bigr)\ +\ \P_{\bar\eta}\bigl(\bar T<n\bigr). \]
It therefore remains to show
\begin{equation}\label{eq:mixtoshow}
    \P_{\bar\eta}\bigl(\tau>n,\bar T\geq n\bigr)\ \leq\ (1-\rho)^n.
\end{equation}
Denoting by $(\bar{\mathfrak F}_n)_{n\geq0}$ the filtration generated by $(\bar X_n)_{n\geq0}$, it holds $\{\tau>n\}=\{\tau>n-1\}\cap\{\bar X_n\notin\Delta\}$ with $\{\tau>n-1\}\in\bar{\mathfrak F}_{n-1}$.
Since further $\{\bar T\geq n\}\in\bar{\mathfrak F}_{n-1}$,
\begin{align*}
\P_{\bar\eta}\bigl(\tau>n,\bar T\geq n\bigr)
\ &=\ \E_{\bar\eta}\Bigl(\P_{\bar\eta}\bigl(\bar X_n\notin\Delta\big|\bar{\mathfrak F}_{n-1}\bigr);\tau>n-1,\bar T\geq n\Bigr) \\
&\leq\ (1-\rho)\, \P_{\bar\eta}\bigl(\tau>n-1,\bar T\geq n-1\bigr),
\end{align*}
where we used that, by the Markov property and $\bar\pi$ realizing \eqref{eq:Dobctrbar}, on $\{\bar T\geq n\}$, we have almost surely
\[ \P_{\bar\eta}\bigl(\bar X_n\notin\Delta\big|\bar{\mathfrak F}_{n-1}\bigr) \ =\ \P_{\bar X_{n-1}}\bigl(\bar X_1\notin\Delta\bigr)\ =\ \bar\pi\bigl(\bar X_{n-1},\Delta^\comp\bigr)\ \leq\ 1-\rho. \]
Iterating this estimate shows \eqref{eq:mixtoshow} and finishes the proof.
\end{proof}

The following elementary estimates can, at times, provide suitable control over exit probabilities.
Similar bounds are used in \cite{Yang23}.

\begin{proposition}\label{prop:stab}
Let $\pi$ be a Markov transition kernel with invariant distribution $\mu\in\calP(\S)$, and $D\subseteq\S$.
Then, it holds
\[ \mathbb P_\mu(T<n)\ \leq\ n\,\mu(D^\comp)\quad\text{for all $n\geq0$,} \]
and, for any $\nu\in\calP(\S)$,
\begin{equation}\label{eq:stabnu}
    \mathbb P_\nu(T<n)\ \leq\ \nu(D^\comp)\ +\ (n-1)\,\sup\nolimits_{x\in D}\pi(x,D^\comp)\quad\text{for all $n\geq1$,}
\end{equation}
where $T$ denotes the first exit time from $D$ of the Markov chain with transition kernel $\pi$.
\end{proposition}

\begin{proof}[Proof of \Cref{prop:stab}]
Define the flow out of $D$ in one transition of $\pi$ initialized in $\eta\in\mathcal P(\mathcal S)$ to be
\begin{equation}\label{eq:flowoutofD}
    Q_\eta(D)\ :=\ \int_D\eta(dx)\pi(x,D^\comp),
\end{equation}
with which it holds that
\begin{equation}\label{eq:exittimeflow}
    \mathbb P_\eta(T<n)\ \leq\ \eta(D^\comp)+\sum_{k=0}^{n-2}Q_{\eta\pi^k}(D)\quad\text{for all $n\geq0$.}
\end{equation}
Indeed, the first term represents the probability of not starting in $D$ in the first place, while the sum bounds the flow out of $D$ in each transition.
For all $k\geq0$, since $\mu$ is invariant with respect to $\pi$, $\mu\pi^k=\mu$ and hence
\[ Q_{\mu\pi^k}(D)\ =\ Q_\mu(D)\ \leq\ \mu\pi(D^\comp)\ =\ \mu(D^\comp). \]
Therefore, from start in $\eta=\mu$, \eqref{eq:exittimeflow} and the fact that $T\geq0$ almost surely yield
\[ \mathbb P_\mu(T<n)\ \leq\ n\,\mu(D^\comp)\quad\text{for all $n\geq0$.} \]
Further, since
\[ Q_\eta(D)\ \leq\ \sup\nolimits_{x\in D}\pi(x,D^\comp)\quad\text{for all $\eta\in\calP(\S)$,} \]
\eqref{eq:exittimeflow} implies that from start in any initial distribution $\nu\in\calP(\S)$,
\[ \mathbb P_\nu(T<n)\ \leq\ \nu(D^\comp)+(n-1)\,\sup\nolimits_{x\in D}\pi(x,D^\comp)\quad\text{for all $n\geq1$.} \]
\end{proof}

We close this section with an illustration of the above results' ability to accurately capture the fast local equilibration of the seldom-trapped chain from start on the circle, see \Cref{fig:seldomtrapped}.

\begin{example}[Local Geometric Mixing of the Seldom-trapped Chain]\label{ex:seldomtrappedbound}
Consider the seldom-trapped chain introduced in \Cref{ex:seldomtrapped}.
On the circle $D=\mathbb S^1$, local Dobrushin contraction holds with rate $1/2$.
Via \Cref{prop:stab}, the exit probabilities from $D$ are for all $n\geq1$ controlled according to
\[ \P_\mu(T<n)\ \leq\ \epsilon+\tfrac12\epsilon^2(n-1)\quad\text{and}\quad\P_x(T<n)\ \leq\ \tfrac12\epsilon^2(n-1)\quad\text{for all $x\in\mathbb S^1$,} \]
where we used \eqref{eq:stabnu} for both estimates to obtain a sharper bound from start in $\mu$.
Therefore, \Cref{cor:mixlocDob} implies
\[ \TV\bigl(\pi^n(x,\cdot),\mu\bigr)\ \leq\ \inf\nolimits_{k\leq n}\bigl(2^{-k} + \epsilon^2(k-1)+\epsilon\bigr)\quad\text{for all $n\geq1$ and $x\in\mathbb S^1$.} \]
This accurately describes rapid local mixing on the circle, see \Cref{fig:seldomtrapped}.
\end{example}

\section{Diffusion Path Monte Carlo and the Proximal Sampler}\label{sec:DPMC}

In this section, we present Diffusion Path Monte Carlo \cite{sifanDP,Hill26}, also known as the Proximal Sampler \cite{Titsias18,Lee21,Sinho22}, that we subsequently study.
Having arisen in different contexts, the equivalence of these two formulations of the same Markov chain is not immediately obvious.
After introducing them separately, we make their equality precise and describe how they originally aimed at different regimes of the dynamics.
From here on, let $\mu$ be a Borel probability measure on $\mathbb R^d$ that we understand as target distribution of sampling.

\subsection{Diffusion Path Monte Carlo}

Proposed in \cite{sifanDP,Hill26}, Diffusion Path Monte Carlo is a discrete-time Markov chain on $\R^d$.
In each transition, it first follows an Ornstein--Uhlenbeck noising diffusion and subsequently follows the noising diffusion's time reversal in law.

Specifically, consider the Ornstein--Uhlenbeck \emph{noising process}
\begin{equation}\label{eq:OU}
    d\Xf_s\ =\ -\tfrac12\Xf_s\,ds\ +\ d\vec B_s,
\end{equation}
where $(\vec B_s)_{s\geq0}$ is a standard Brownian motion in $\mathbb R^d$.
Its Markov transition semigroup is
\begin{equation}\label{eq:Pf}
    \Pf_s(x,\cdot)\ =\ \mathcal N\bigl(e^{-s/2}x,(1-e^{-s})I_d\bigr).
\end{equation}
The noising process quickly transforms the target distribution towards a standard normal distribution, namely, if $\Xf_0\sim\mu$, its law at time $s$ is
\begin{equation}\label{eq:mus}
    \mu_s\ :=\ \mu\Pf_s\ =\ \bigl((e^{-s/2})_\#\mu\bigr)\ast\mathcal N\bigl(0,(1-e^{-s})I_d\bigr),
\end{equation}
where $(e^{-s/2})_\#\mu$ denotes the pushforward of $\mu$ by $x\mapsto e^{-s/2}x$.
Note that, for $s>0$, $\mu_s$ is absolutely continuous with respect to Lebesgue measure, with positive and smooth Lebesgue density that we also denote by $\mu_s$.

For a fixed \emph{noising time} $t>0$, the time-reversed process
\[ \Xb^{(t)}_s\ :=\ \Xf_{t-s}\quad\text{for $0\leq s\leq t$} \]
is a Markov process and, on $0\leq s<t$, a weak solution of the \emph{denoising equation}
\begin{equation}\label{eq:Xb}
    d\Xb^{(t)}_s\ =\ \left(\tfrac12\Xb^{(t)}_s+\nabla\log\mu_{t-s}(\Xb^{(t)}_s)\right)ds\ +\ d\cev B_s,
\end{equation}
for a Brownian motion $(\cev B_s)_{s\geq0}$, see \cite{Haussmann86}.
See also \cite{Nelson67,Anderson82,Foellmer85,Foellmer86,Cattiaux23} for historical background and further developments.
We denote by $\Pb_t$ the full-horizon transition kernel of the reversed process, that is,
\[ \Pb_t(x,\cdot)\ :=\ \Law\bigl(\Xb^{(t)}_t\big|\Xb^{(t)}_0=x\bigr)\ =\ \Law(\Xf_0\mid\Xf_t=x), \]
which, by Bayes' theorem, may be written as
\begin{equation}\label{eq:Pb}
    \Pb_t(x,dy)\ =\ \frac{\Pf_t(y,x)}{\mu_t(x)}\,\mu(dy)\quad\text{for all $x\in\R^d$,}
\end{equation}
where $\Pf_t(y,x)$ denotes the Lebesgue density of $\Pf_t(y,dx)$.
In particular, $\mu_t\Pb_t=\mu$.

Combining the transition kernels of the noising and denoising processes yields a transition of Diffusion Path Monte Carlo.

\begin{definition}\label{def:DPMC}
For fixed noising time $t>0$, the Markov transition kernel of \emph{ideal Diffusion Path Monte Carlo} takes the form
\[ \pi_{(t)}(x,\cdot)\ :=\ \Pf_t\Pb_t(x,\cdot)\ =\ \int_{\R^d}\Pf_t(x,dy)\Pb_t(y,\cdot)\quad\text{for all $x\in\R^d$.} \]
\end{definition}

\begin{remark}[Invariance and Reversibility]
\emph{Invariance} of $\mu$ under $\pi_{(t)}$ is immediate from the definition as
\[ \mu\pi_{(t)}\ =\ \mu\Pf_t\Pb_t\ =\ \mu_t\Pb_t\ =\ \mu. \]
Beyond invariance, $\pi_{(t)}$ is \emph{reversible} with respect to $\mu$, meaning that \emph{detailed balance}
$\mu(dx)\,\pi_{(t)}(x,dy)=\mu(dy)\,\pi_{(t)}(y,dx)$ as measures on $\R^d\times\R^d$
holds, or equivalently, that $\pi_{(t)}$ is a \emph{self-adjoint} operator on $L^2(\mu)$, that is,
$\bigl\langle f,\pi_{(t)}g\bigr\rangle_{L^2(\mu)} = \bigl\langle\pi_{(t)}f,g\bigr\rangle_{L^2(\mu)}$ for all $f,g\in L^2(\mu)$.
The self-adjointness of $\pi_{(t)}$ immediately follows from the fact that the operators
\[ \Pf_t:L^2(\mu_t)\to L^2(\mu)\quad\text{and}\quad\Pb_t:L^2(\mu)\to L^2(\mu_t) \]
are adjoint, i.e., $\Pf_t^\ast=\Pb_t$, where
$\bigl\langle f,\Pf_tg\bigr\rangle_{L^2(\mu)} = \bigl\langle\Pf_t^\ast f,g\bigr\rangle_{L^2(\mu_t)}$ for all $f\in L^2(\mu)$ and $g\in L^2(\mu_t)$,
so that $\pi_{(t)}=\Pf_t\Pf_t^\ast$.
\end{remark}

The transition kernel $\pi_{(t)}$ is called \emph{ideal} since it involves the exact denoising kernel $\Pb_t$, which is not known exactly in practice.
To obtain an implementable sampling method, $\Pb_t$ may be approximated by discretizing the denoising equation \eqref{eq:Xb} using a learned approximation of the score function $\nabla\log\mu_{t-s}$.
The bias introduced by discretization and score approximation can subsequently be corrected via a Metropolis--Hastings construction.
We present the details of the resulting \emph{Metropolis-adjusted Diffusion Path Monte Carlo} method, put forward in \cite{sifanDP,Hill26}, in \Cref{sec:MDPMC} and study its mixing in \Cref{sec:MDPMCmix}.

\subsection{The Proximal Sampler}

With step size $h>0$, the Proximal Sampler \cite{Titsias18,Lee21,Sinho22} performs Gibbs sampling in the extended probability distribution
\[ \bar\mu^{(h)}(dx\,dx')\ :=\ \mu(dx)\,\N(x,hI_d)(dx')\quad\text{on $\R^{2d}$.} \]
To write the transitions explicitly, consider the regular versions of the conditional distributions
\[ \bar\mu^{(h)}(dx'|x)\ =\ \N(x,hI_d)(dx')\quad\text{and}\quad\bar\mu^{(h)}(dx|x')\ =\ \frac{\mu(dx)\phi_h(x'-x)}{\int_{\R^d}\mu(dy)\phi_h(x'-y)}, \]
where $\phi_h(y):=(2\pi h)^{-d/2}\exp(-|y|^2/(2h))$.
A transition of the Proximal Sampler from state $x$ consists of a \emph{forward step} $X'\sim\bar\mu^{(h)}(\,\cdot\,|x)$, followed by a \emph{backward step} to its next state $X\sim\bar\mu^{(h)}(\,\cdot\,|X')$.

\begin{definition}\label{def:PS}
For fixed step size $h>0$, the Markov transition kernel of \emph{the Proximal Sampler} takes the form
\[ \pi^{(h)}(x,\cdot)\ :=\ \int_{\R^d}\bar\mu^{(h)}(dx'|x)\,\bar\mu^{(h)}(\,\cdot\,|x')\quad\text{for all $x\in\R^d$.} \]
\end{definition}

Reversibility with respect to $\mu$ of the Proximal Sampler as a Gibbs sampler is classical, see \cite{Liu94}.

\subsection{Two Sides of the Same Coin: Small vs. Large Noising Time}

Even though the forward/backward step and noising/denoising kernels in the Proximal Sampler and ideal Diffusion Path Monte Carlo are not equal, for suitably related step size and noising time, they coincide up to rescaling, more precisely, up to pushforward, respectively pullback, by the dilation $f(x)=e^{t/2}x$, see \eqref{eq:sameid}.
As recorded in the following proposition, a change of variables shows their transition kernels to be equal.

\begin{proposition}\label{prop:same}
The transition kernels $\pi_{(t)}$ of ideal Diffusion Path Monte Carlo with noising time $t>0$, see \Cref{def:DPMC}, and $\pi^{(h)}$ of the Proximal Sampler with step size $h>0$, see \Cref{def:PS}, coincide in the sense that
\[ \pi_{(t)}\ =\ \pi^{(h)}\quad\text{for}\quad t\ =\ \log(h+1). \]
\end{proposition}

\begin{proof}[Proof of \Cref{prop:same}]
Let $t=\log(h+1)$ and $f(y):=e^{t/2}y$ for any $y\in\R^d$.
We write $f_\#\pi(x,B):=\pi(x,f^{-1}(B))$, where $x\in\R^d$ and $B\subseteq\R^d$ Borel, for the pushforward of a stochastic kernel $\pi$ on $\R^d$ by $f$.
The forward/backward step and noising/denoising kernels in the Proximal Sampler and ideal Diffusion Path Monte Carlo coincide up to pushforward, respectively pullback, by $f$, that is, for all $x,x'\in\R^d$,
\begin{equation}\label{eq:sameid}
    \bar\mu^{(h)}(\,\cdot\,|x)\ =\ f_\#\Pf_t(x,\cdot)\quad\text{and}\quad\bar\mu^{(h)}(dx|x')\ =\ \Pb_t\bigl(f^{-1}(x'),dx\bigr).
\end{equation}
The first identity holds as both sides equal $\N(x,hI_d)$, cf. \eqref{eq:Pf}.
The second one is shown below.
The previous display together with a change of variables shows the desired equality of transition kernels, namely that
\begin{align*}
\pi^{(h)}(x,\cdot)
\ &=\ \int_{\R^d}\bar\mu^{(h)}(dx'|x)\,\bar\mu^{(h)}(\,\cdot\,|x')
\ =\ \int_{\R^d}f_\#\Pf_t(x,dx')\Pb_t\bigl(f^{-1}(x'),\cdot\bigr) \\
&=\ \int_{\R^d}\Pf_t(x,dx')\Pb_t(x',\cdot)\ =\ \pi_{(t)}(x,\cdot)\quad\text{for all $x\in\R^d$.}
\end{align*}
We are left to show the second identity in \eqref{eq:sameid}.
For all $x\in\R^d$, $f_\#\Pf_t(x,\cdot)$ is absolutely continuous with respect to Lebesgue measure with density $f_\#\Pf_t(x,x')=\phi_h(x'-x)$, and so is the $\mu$-average of its first component, whose Lebesgue density we denote by $\mu(f_\#\Pf_t)$.
Hence, for all $x'\in\R^d$,
\[ \bar\mu^{(h)}(dx|x')\ =\ \frac{\mu(dx)\phi_h(x'-x)}{\int_{\R^d}\mu(dy)\phi_h(x'-y)}\ =\ \frac{\mu(dx)f_\#\Pf_t(x,x')}{\mu(f_\#\Pf_t)(x')}\ =\ \frac{f_\#\Pf_t(x,x')}{f_\#\mu_t(x')}\mu(dx), \]
where we used that $\mu(f_\#\Pf_t)=f_\#(\mu\Pf_t)=f_\#\mu_t$ by \eqref{eq:mus}.
As $f_\#\nu(x)=e^{-dt/2}(\nu\circ f^{-1})(x)$ is the Lebesgue density of the pushforward of any $\nu(dx)=\nu(x)\,dx$ by $f$, it holds by \eqref{eq:Pb} that
\[ \bar\mu^{(h)}(dx|x')\ =\ \frac{\Pf_t\bigl(x,f^{-1}(x')\bigr)}{\mu_t\circ f^{-1}(x')}\mu(dx)\ =\ \Pb_t\bigl(f^{-1}(x'),dx\bigr)\quad\text{for all $x'\in\R^d$.} \]
\end{proof}

The above proposition identifies the Proximal Sampler and ideal Diffusion Path Monte Carlo as two sides of the same coin, or rather Markov chain.
Their distinction arose from different practical considerations of their implementation, aiming at different regimes of said Markov chain---small vs. large step size or noising time.

On the one hand, in \cite{Lee21} putting forth the Proximal Sampler, the non-trivial backward sampling step from $\bar\mu^{(h)}(\,\cdot\,|x')$, also called \emph{restricted Gaussian oracle}, is proposed to be implemented via rejection sampling.
Assuming $\mu(dx)=\mu(x)\,dx$, therein, the term $\log\mu(x)$ in $\bar\mu^{(h)}(dx|x')\propto e^{\log\mu(x)-\frac{|x'-x|^2}{2h}}\,dx$ is approximated by $\log\mu(x')+\nabla\log\mu(x')\cdot(x-x')$, which, up to normalization, yields the easy to sample measure $\mathcal N\bigl(x'+h\nabla\log\mu(x'),hI_d\bigr)$.
To obtain a sample from the original backward step, samples from this normal distribution are proposed and accepted with a suitable probability one after another until a sample has been accepted.
For strongly log-concave $\mu$ with Lipschitz continuous log-density, \cite{Lee21} shows this strategy to produce an approximate sample from the backward step in, on average, finitely many iterations, assuming a sufficiently small step size $h$.
Subsequent works extend and improve the implementation of the backward step, see \cite{Sinho22,Liang22,Liang23,Fan23,Altschuler24,Huang24,Rakhlin26,Takagi26}.
Importantly, all these techniques require a sufficiently small step size and thus aim the resulting dynamics at \emph{small step sizes}.

On the other hand, Diffusion Path Monte Carlo, as proposed in \cite{sifanDP,Hill26}, intends to capitalize on the ability to approximately learn the score function $\nabla\log\mu_{t-s}$ appearing in the denoising equation \eqref{eq:Xb}.
The particular intention therein is to choose the noising time $t$ sufficiently large as to transform the target distribution $\mu$ into a noised measure $\mu_t$ close to the standard Gaussian distribution.
Intuitively, this enables the ideal dynamics, assuming access to the exact denoising process, to make global moves throughout general targets and overcome sampling bottlenecks faced by many other Markov chain Monte Carlo methods.
The therefore required noising times $t$, or, equivalently, Proximal Sampler step sizes $h=e^t-1$, are much larger than the small step sizes suitable for the previously discussed rejection sampling implementations.
In particular, in contrast to the original Proximal Sampler, Diffusion Path Monte Carlo is aimed at such \emph{large step sizes or noising times}.
In the sext section, we make this intuition rigorous.
\newpage

\section{Mixing of Ideal Diffusion Path Monte Carlo and the Proximal Sampler}\label{sec:DPMCmix}

In this section, we study the mixing of ideal Diffusion Path Monte Carlo and the Proximal Sampler.
Therein, in \Cref{sec:previous}, we first give an overview over the considerable previous work on their convergence to equilibrium, identifying a vacancy in the existing literature: mixing for large noising times or step sizes in general target distributions.
More precisely, we discover that, while small noising times already are well-understood, for large noising times, the convergence regardless of target, as intuited above, remains open.

In \Cref{sec:DPMCmixsub}, we fill this gap by providing mixing guarantees for large noising times without imposing any assumptions on the target distribution beyond being a probability measure.
Besides their generality regarding target distributions, the obtained mixing time bounds feature very mild sensitivity to initialization, holding uniformly across starts in a fixed ball.
The proof relies on the approach to local geometric mixing via Dobrushin contraction presented in \Cref{sec:locDobmix}.

Subsequently, in \Cref{sec:gapmix}, we show how bounds on the spectral gap of ideal Diffusion Path Monte Carlo and the Proximal Sampler, as developed in previous work, can be used to control mixing times from general initial distributions, instead of only from warm starts.

\subsection{Previous Work}\label{sec:previous}

\subsubsection{Small Noising Times or Step Sizes}

We start by reviewing the existing results on convergence to equilibrium of the Proximal Sampler, equivalently stated in terms of the transition kernel $\pi_{(t)}$ of ideal Diffusion Path Monte Carlo.

For strongly log-concave target distributions, meaning that $\mu(dx)=\mu(x)\,dx$ with $-\log\mu(x)$ being $m$-strongly convex for $m>0$, \cite{Lee21} shows contraction in the standard $L^2$-Wasserstein distance, which, in particular, yields the geometric contraction towards equilibrium
\[ \W_2\bigl(\nu\pi_{(t)},\mu\bigr)\ \leq\ \left(1-\left(1+\frac{m^{-1}}{e^t-1}\right)^{-1}\right)\W_2(\nu,\mu)\quad\text{for all $t>0$ and $\nu\in\calP(\R^d)$.} \]

Beyond strongly log-concave targets, \cite{Sinho22} shows both contraction toward equilibrium in relative entropy for targets satisfying a log-Sobolev inequality, and in the $\chi^2$-divergence for targets satisfying a Poincar\'e inequality.
More precisely, we say that $\mu$ satisfies a log-Sobolev inequality with constant $C_{LS}>0$, if
\begin{equation}\label{eq:LSI}
    \operatorname{Ent}_\mu(f^2)\ \leq\ 2C_{LS}\int_{\R^d}|\nabla f|^2\,d\mu\quad\text{for all $f\in C^\infty_c(\R^d)$,}
\end{equation}
where $C^\infty_c(\R^d)$ denote the smooth and compactly supported functions, and $\operatorname{Ent}_\mu(f):=\int_{\R^d}f\log\bigl(f/\int_{\R^d}f\,d\mu\bigr)\,d\mu$.
In such target distributions, \cite{Sinho22} shows
\[ H\bigl(\nu\pi_{(t)}\big|\mu\bigr)^{1/2}\ \leq\ \left(1-\left(1+\frac{C_{LS}}{e^t-1}\right)^{-1}\right)H(\nu|\mu)^{1/2}\quad\text{for all $t>0$ and $\nu\in\calP(\R^d)$,} \]
where the relative Entropy, also Kullback--Leibler divergence, to equilibrium from any $\eta\in\calP(\R^d)$ is defined by
\[ H(\eta|\mu)\ :=\
\begin{cases}
\int_{\R^d}\frac{d\eta}{d\mu}\log\frac{d\eta}{d\mu}\,d\mu & \text{if $\eta\ll\mu$,} \\
\infty & \text{otherwise.}
\end{cases} \]
More generally, assuming the target to satisfy a Poincar\'e inequality with constant $C_P>0$, that is,
\begin{equation}\label{eq:PI}
    \Var_\mu(f)\ \leq\ C_P\int_{\R^d}|\nabla f|^2\,d\mu\quad\text{for all $f\in C^\infty_c(\R^d)$},
\end{equation}
where $\Var_\mu(f):=\int_{\R^d}\bigl(f-\int_{\R^d}f\,d\mu\bigr)^2\,d\mu$, \cite{Sinho22} shows
\[ \chi^2\bigl(\nu\pi_{(t)}\big|\mu\bigr)^{1/2}\ \leq\ \left(1-\left(1+\frac{C_P}{e^t-1}\right)^{-1}\right)\chi^2(\nu|\mu)^{1/2}\quad\text{for all $t>0$ and $\nu\in\calP(\R^d)$,} \]
where the $\chi^2$-divergence to equilibrium from any $\eta\in\calP(\R^d)$ is
\[ \chi^2(\eta|\mu)\ :=\
\begin{cases}
\bigl\|\frac{d\eta}{d\mu}-1\bigr\|_{L^2(\mu)}^2 & \text{if $\eta\ll\mu$,} \\
\infty & \text{otherwise.}
\end{cases} \]

Note that the contraction toward equilibrium in the $\chi^2$-divergence of ideal Diffusion Path Monte Carlo and the Proximal Sampler is equivalent to its right spectral gap being lower bounded according to
\begin{equation}\label{eq:gapPI}
   \Gap\bigl(\pi_{(t)}\bigr)\ :=\ 1- \sup_{f\in L^2_0(\mu), f\ne0}\frac{\langle f,\pi_{(t)} f\rangle_{L^2(\mu)}}{\|f\|_{L^2(\mu)}^2}\ \geq\ \left(1+\frac{C_P}{e^t-1}\right)^{-1},
\end{equation}
where $L^2_0(\mu)$ denotes the space of square integrable functions with zero mean.
Indeed, the absolute spectral gap $1-\|\pi_{(t)}\|_{L^2_0(\mu)\to L^2_0(\mu)}$, which corresponds to the $\chi^2$-contraction coefficient toward equilibrium, equals the right spectral gap due to the fact that the transition operator of ideal Diffusion Path Monte Carlo and the Proximal Sampler, beyond being self-adjoint, is also positive semidefinite, that is,
\[ \bigl\langle f,\pi_{(t)}f\bigr\rangle_{L^2(\mu)}\ =\ \bigl\langle f,\Pf_t\Pf_t^\ast f\bigr\rangle_{L^2(\mu)}\ =\ \bigl\|\Pf_t^\ast f\bigr\|_{L^2(\mu_t)}^2\ \geq\ 0\quad\text{for all $f\in L^2(\mu)$.} \]

Interpolating between log-Sobolev and Poincar\'e inequalities, \cite{Sinho22} further consider target distributions satisfying Latala–Oleszkiewicz inequalities, and generalize their results to R\'enyi divergences.
Moreover, \cite{Andre26} studies convergence to equilibrium in terms of the relative Fisher information in strongly log-concave target distributions, and \cite{Sid25} in terms of general $\Phi$-divergences.

Even though the above results hold for all noising times or step sizes, they are effectively aimed at short noising times due to the imposed assumptions on the target distribution.
This is illustrated by \cite{Sinho22} showing the Proximal Sampler to follow an entropy-regularized Jordan--Kinderlehrer--Otto scheme \cite{Jordan98}, which, for small step sizes, approximates the Wasserstein gradient flow of the relative entropy in the space of probability measures, see \cite{Capdeville26}.
Consequently, for small step sizes, the dynamics approximate the overdamped Langevin diffusion, whose marginal distributions precisely follow said gradient flow.
In particular, in this step size regime, the Proximal Sampler is expected to make local moves and suffer metastability in multimodal target distributions as well as sub-geometric convergence to equilibrium in the presence of heavy tails, similarly to the overdamped Langevin diffusion, see \cite{Bovier04,Balasubramanian24}.
Therefore, assumptions on the target as above, imposing restrictions on such features, are natural in this regime.

\subsubsection{Large Noising Times or Step Sizes}

A recent result on the convergence to equilibrium of ideal Diffusion Path Monte Carlo and the Proximal Sampler aiming at large noising times or step sizes is due to \cite{sifanDP} and shows that, if $\mu=\sum_{i=1}^Nw_i\mu_i$ is a mixture of $N\geq1$ components $\mu_i\in\calP(\R^d)$, each satisfying the Poincar\'e inequality \eqref{eq:PI}, with non-negative weights $\sum_{i=1}^Nw_i=1$, then
\begin{equation}\label{eq:gapPImix}
    \Gap\bigl(\pi_{(t)}\bigr)\ \geq\ \left(1+\frac{C_P}{e^t-1}\right)^{-1}\exp\left(-\left(\frac{\max_{1\leq i,j\leq N}\W_2(\mu_i,\mu_j)}{2\sqrt{e^t-1}}\right)^2\right).
\end{equation}
The mixture assumption allows for multimodal targets with strongly separated modes.
For sufficiently large noising times $t\gtrsim\log\bigl(\max_{i,j}\W_2(\mu_i,\mu_j)^2+1\bigr)$, this multimodality slows the $\chi^2$-convergence at most by an absolute constant compared to targets satisfying the Poincar\'e inequality globally.
In contrast, shorter noising times yield the spectral gap lower bound to degenerate exponentially, reflecting the expected metastability.

\subsection{Local Mixing for Large Noising Times in General Target Distributions}\label{sec:DPMCmixsub}

In regards to controlling mixing times of ideal Diffusion Path Monte Carlo and the Proximal Sampler for large noising times or step sizes, the spectral gap bound \eqref{eq:gapPImix} of \cite{sifanDP} suffers two crucial restrictions:
On the one hand, while allowing targets with strongly separated modes, it limits the individual components to satisfy a Poincar\'e inequality.
This, for instance, requires sub-exponential tails, ruling out heavier-tailed distributions.
On the other hand, the resulting convergence to equilibrium in $\chi^2$-divergence only controls the mixing time from warm starts.
More precisely, it holds for all $\eps>0$ and $\beta>0$ that
\begin{equation}\label{eq:mixwarm}
    \tmix(\eps,\nu)\ \leq\ \Bigl\lceil\Gap\bigl(\pi_{(t)}\bigr)^{-1}\log\bigl(\beta\,\eps^{-1}\bigr)\Bigr\rceil\quad\text{for all $\beta$-warm $\nu$,}
\end{equation}
where an initial distribution $\nu\in\calP(\R^d)$ is \emph{$\beta$-warm} if $\chi^2(\nu|\mu)^{1/2}\leq\beta$.
This directly follows from the fact that $\TV(\eta,\mu)\leq\frac12\chi^2(\eta|\mu)^{1/2}$ for all $\eta\in\calP(\R^d)$, and that $\pi_{(t)}$ contracts the $\chi^2$-divergence to equilibrium at rate $\Gap(\pi_{(t)})$.
In particular, a lower bound on the spectral gap, on its own, does not control general mixing times since initialization in point masses is not warm, that is, $\chi^2(\delta_x|\mu)^{1/2}=\infty$ for any $x\in\R^d$.

In this section, we obtain bounds on mixing times of ideal Diffusion Path Monte Carlo and the Proximal Sampler for large noising times or step sizes that do not suffer these two restrictions.
Notably, we do not impose any assumptions on the target or initial distributions beyond being probability measures.

Our results are based on the approach to local mixing via Dobrushin contraction, presented in \Cref{sec:locDobmix}.
To derive local Dobrushin contraction, we use that the total variation distance is non-expansive under the action of any Markov transition kernel.
Specifically, it holds for any such kernel $P$ on $\bigl(\R^d,\mathfrak B(\R^d)\bigr)$ that
\[ \TV(\eta P,\tilde\eta P)\ \leq\ \TV(\eta,\tilde\eta)\quad\text{for all $\eta,\tilde\eta\in\calP(\R^d)$.} \]
This is an immediate consequence of the dual representation of total variation
\[ \mathrm{TV}(\eta,\tilde\eta)\ =\ \tfrac12\sup\nolimits_{\|f\|_\infty\leq1}\bigl|\eta(f)-\tilde\eta(f)\bigr| \]
and the fact that $\|Pf\|_\infty\leq\|f\|_\infty$ for all measurable $f$.
Since $\pi_{(t)}=\Pf_t\Pb_t$ decomposes into the noising and denoising kernels, the previous display yields
\[ \TV\bigl(\pi_{(t)}(x,\cdot),\pi_{(t)}(\tilde x,\cdot)\bigr)\ \leq\ \TV\bigl(\Pf_t(x,\cdot),\Pf_t(\tilde x,\cdot)\bigr). \]
This estimate considerably simplifies the Dobrushin contraction, replacing the total variation distance between the full transition kernels $\pi_{(t)}$ by the distance between the noising kernels $\Pf_t(x,\cdot)=\mathcal N\bigl(e^{-t/2}x,(1-e^{-t})I_d\bigr)$.
As shown in \cite[Thm.~1]{Barsov86}, the latter evaluates to
\begin{equation}\label{eq:TV_Pf}
    \TV\bigl(\Pf_t(x,\cdot),\Pf_t(\tilde x,\cdot)\bigr)\ =\ 2\Phi\left(\frac{|x-\tilde x|}{2\sqrt{e^t-1}}\right)-1\ \leq\ \frac{|x-\tilde x|}{\sqrt{2\pi(e^t-1)}},
\end{equation}
where $\Phi$ denotes the cumulative distribution function of the standard normal distribution, and the inequality uses that $\Phi(r)-1/2=(2\pi)^{-1/2}\int_0^re^{-s^2/2}ds\leq r/\sqrt{2\pi}$ for all $r\geq0$.
Combining the previous two displays and taking the supremum shows
\begin{equation}\label{eq:locDob_DPMC}
    \sup_{x,\tilde x\in D}\TV\bigl(\pi_{(t)}(x,\cdot),\pi_{(t)}(\tilde x,\cdot)\bigr)\ \leq\ \frac{\diam(D)}{2\sqrt{e^t-1}}
\end{equation}
and hence local Dobrushin contraction for sufficiently large noising times.
Inserting this estimate into the assertion of \Cref{cor:mixlocDob} immediately yields the following local geometric mixing guarantee for ideal Diffusion Path Monte Carlo and the Proximal Sampler at large noising times or step sizes.

\begin{theorem}\label{thm:mixDPMC}
Let $\pi_{(t)}$ be the transition kernel with invariant distribution $\mu\in\calP(\R^d)$ of ideal Diffusion Path Monte Carlo with noising time $t>0$, or equivalently of the Proximal Sampler with step size $h=e^t-1$.
Then, it holds for all $D\subseteq\R^d$, initial distributions $\nu\in\calP(\R^d)$, and $n\geq0$ that
\[ \TV\bigl(\nu\pi_{(t)}^n,\mu\bigr)\ \leq\ \left(\frac{\diam(D)}{2\sqrt{e^t-1}}\right)^n+\ \P_{\nu}(T<n)\ +\ \P_{\mu}(T<n), \]
where $T$ denotes the first exit time from $D$.
\end{theorem}

\subsubsection{Local Mixing in One Step}

Setting $n=1$ in the theorem, since $\P_\eta(T<1)=\eta(D^\comp)$ for all $\eta\in\calP(\R^d)$, we have
\begin{equation}\label{eq:DPMC_onestep}
    \TV\bigl(\nu\pi_{(t)},\mu\bigr)\ \leq\ \frac{\diam(D)}{2\sqrt{e^t-1}}\ +\ \nu(D^\comp)\ +\ \mu(D^\comp).
\end{equation}
This yields local mixing in a single transition for suitably large noising times $t$.
To make this precise, we need to specify a domain which encompasses sufficient mass of both $\mu$ and $\nu$, whose diameter then determines $t$.
Therefore, denoting by $B_R$ the closed centered ball of radius $R>0$ in $\R^d$, for $\delta>0$, define the \emph{centered partial diameter} of a probability measure $\eta\in\calP(\R^d)$ by
\begin{equation}\label{eq:diam}
    \diam(\eta,\delta)\ :=\ 2\,\inf\bigl\{R\geq0:\eta(B_R^\comp)\leq\delta\bigr\}.
\end{equation}
Allowing for arbitrary Borel sets instead of centered balls yields the \emph{partial diameter} commonly defined in metric measure geometry, see \cite{Shioya16}.
We restrict to centered balls for simplicity.
Choosing $D$ to be such a ball of diameter $\max\bigl(\diam(\mu,\tfrac\eps3),\diam(\nu,\tfrac\eps3)\bigr)$, for which the last two terms on the right hand side of \eqref{eq:DPMC_onestep} are each bounded above by $\eps/3$, and considering arbitrary initial states in a centered ball of radius $R$ yields the following mixing time bound.

\begin{corollary}[Local Mixing in One Step]\label{cor:oneDPMC}
Let $\pi_{(t)}$ be the transition kernel with invariant distribution $\mu\in\calP(\R^d)$ of ideal Diffusion Path Monte Carlo with noising time $t>0$, or equivalently of the Proximal Sampler with step size $h=e^t-1$.
For any $\eps>0$ and $R>0$, if
\begin{equation}\label{eq:oneDPMCt}
    t\ \geq\ \log\left(\Bigl(3\max\bigl(\tfrac12\diam(\mu,\tfrac\eps3),R\bigr)\eps^{-1}\Bigr)^2+1\right),
\end{equation}
then $\tmix(\eps,B_R)=1$.
\end{corollary}

Notably, the centered partial diameter $\diam(\mu,\tfrac\eps3)$ is finite for all $\mu\in\calP(\R^d)$ and $\eps>0$ by continuity from below, that is, $\lim_{R\to\infty}\mu(B_R)=1$.
Therefore, the corollary asserts local mixing in one single step of ideal Diffusion Path Monte Carlo and the Proximal Sampler \emph{for any target probability distribution $\mu\in\calP(\R^d)$}, given that the noising time or step size is sufficiently large.

It is further worth emphasizing the mild dependence of the result on initialization, the sufficient noising time only growing logarithmically in $R$ while guaranteeing mixing in one step \emph{uniformly from start in any $x\in B_R$}.
In contrast, as discussed at the beginning of this section, see \eqref{eq:mixwarm}, mixing time bounds resulting from spectral gaps, and similarly from relative entropy contraction toward equilibrium, see \Cref{sec:previous}, suffer a far more severe sensitivity on initialization, requiring a suitably warm start.
However, in \Cref{sec:gapmix} below, we describe the ability of $\pi_{(t)}$ for sufficiently large noising time to generate a warm start in one transition.
This allows to extend spectral gap based mixing time bounds to general initial distributions.

According to \eqref{eq:oneDPMCt}, noising times sufficient for mixing to accuracy $\eps$ in one step uniformly from start in $B_R$ can be chosen of order
\begin{equation}\label{eq:suffmixone}
    \mathcal O\Bigl(\log\max\bigl(\diam(\mu,\tfrac\eps3),\eps^{-1},R,e\bigr)\Bigr).
\end{equation}
The dependence of the noising time on the inverse accuracy is discussed in detail in \Cref{ex:tails}.

\begin{remark}[Dependence on Partial Diameter]
We deduce local Dobrushin contraction of $\pi_{(t)}$ for suitable $t$, see \eqref{eq:locDob_DPMC}, from contraction of the noising kernel $\Pf_t$, see \eqref{eq:TV_Pf}.
This contraction describes the noising kernel's ability to locally forget its initial state, which yields mixing in one step as stated in \Cref{cor:oneDPMC}.
Intuitively, $\delta_x\Pf_t\approx\mu_t$ for all $x\in B_R$, based on which denoising produces a sample from $\pi_{(t)}(x,\cdot)=\delta_x\Pf_t\Pb_t\approx\mu_t\Pb_t=\mu$.

Closely related to how the noising kernel forgets its initial state is how the target distribution $\mu$ is forgotten, that is, the convergence of $\mu_t=\mu\Pf_t$ to $\mu_\infty:=\mathcal N(0,I_d)$.
For target distributions having non-negligible mass of order $\eps$ at distance $D/2$ from the origin, \cite{Miha25} quantifies necessary and sufficient noising times for $\TV(\mu_t,\mu_\infty)$ to decrease below $\eps$.
Specifically, for suitably large $D>0$ and small $\eps,\delta>0$, assuming the existence of $x\in\R^d$ with $|x|=\frac12D(1+\delta)$ such that
\begin{equation}\label{eq:data}
    \mu\bigl(B_{\frac12\delta D}(x)\bigr)\ =\ 4\eps\quad\text{and}\quad\mu\bigl(B_{\frac12(1+2\delta)D}\bigr)\ \geq\ 1-\frac\eps3,
\end{equation}
where $B_r(x)$ denotes the ball of radius $r$ centered in $x$, it is shown that
\[ \TV\left(\mu_{2\log(\frac12D)-\Theta(\log\log\eps^{-1})},\mu_\infty\right)\ \geq\ \frac32\,\eps\quad\text{while}\quad \TV\left(\mu_{2\log(\frac12D)+\Theta(\log\eps^{-1})},\mu_\infty\right)\ <\ \eps. \]

Note that \eqref{eq:data} implies $D\leq\diam(\mu,\tfrac\eps3)\leq(1+2\delta)D$ and hence $\diam(\mu,\tfrac\eps3)\asymp D$.
In this case, \Cref{cor:oneDPMC} establishes mixing in one step to accuracy $\eps$ uniformly from start in $B_R$ for sufficiently large noising times
\[ t\ =\ \mathcal O\Bigl(\log\max\bigl(D,\eps^{-1},R,e\bigr)\Bigr), \]
see \eqref{eq:suffmixone}.
The results of \cite{Miha25} therefore support the logarithmic dependence on the centered partial diameter and are consistent with the logarithmic dependence on inverse accuracy.
\end{remark}

\begin{remark}
The bound \eqref{eq:DPMC_onestep} on the total variation to equilibrium after one step can be obtained directly from \eqref{eq:locDob_DPMC}.
Indeed, using the invariance of $\mu$ with respect to $\pi_{(t)}$ and joint convexity of total variation, splitting the domain of integration into $D\times D$ and its complement, and inserting \eqref{eq:locDob_DPMC} together with the trivial bound $\TV\leq1$, shows
\begin{align*}
\TV\bigl(\nu\pi_{(t)},\mu\bigr)
\ &=\ \TV\bigl(\nu\pi_{(t)},\mu\pi_{(t)}\bigr)
\ \leq\ \int_{\R^d\times\R^d}\TV\bigl(\pi_{(t)}(x,\cdot),\pi_{(t)}(\tilde x,\cdot)\bigr)\,(\nu\otimes\mu)(dx\,d\tilde x) \\
&\leq\ \sup_{x,\tilde x\in D}\TV\bigl(\pi_{(t)}(x,\cdot),\pi_{(t)}(\tilde x,\cdot)\bigr)+(\nu\otimes\mu)\bigl((D\times D)^\comp\bigr) \\
&\leq\ \frac{\diam(D)}{2\sqrt{e^t-1}}+\nu(D^\comp)+\mu(D^\comp).
\end{align*}
Beyond this one-step bound, \Cref{thm:mixDPMC} controls total variation to equilibrium after $n>1$ transitions.
\end{remark}

\subsubsection{Local Geometric Mixing}

To derive mixing time bounds from the assertion of \Cref{thm:mixDPMC} for $n>1$ transitions, we need to control the in this case non-trivial exit probabilities from $D$.
The following result combines the above theorem with such control, obtained under the additional assumption on the target distribution to have finite variance.
While this restricts target distributions only slightly, it excludes some heavy-tailed ones, see \Cref{ex:tails}.

\begin{theorem}[Local Geometric Mixing]\label{thm:DPMCcoupmix}
Let $\pi_{(t)}$ be the transition kernel with invariant distribution $\mu\in\calP(\R^d)$ of ideal Diffusion Path Monte Carlo with noising time $t>0$, or equivalently of the Proximal Sampler with step size $h=e^t-1$.
Assume $\int_{\R^d}x\,\mu(dx)=0$.
Then, it holds for all $R>0$ and $n\geq0$ that
\begin{equation}\label{eq:DPMCcoupmix}
    \sup_{x\in B_R}\TV\bigl(\pi_{(t)}^n(x,\cdot),\mu\bigr)\ \leq\ \left(\frac{R}{\sqrt{e^t-1}}\right)^n+2n\,\mu(B_R^\comp)^{1/2}\exp\left(\frac{\Var(\mu)+R^2}{e^t-1}\right),
\end{equation}
where $\Var(\mu)=\int_{\R^d}|x|^2\mu(dx)$.
In particular, for all $\eps>0$, if
\begin{equation}\label{eq:DPMCcoupmixt}
    t\ \geq\ \log\left(\Var(\mu)+e^2\max\Bigl(\tfrac12\diam\bigl(\mu,\tfrac{\eps^2}{(4e\lceil\log(2\eps^{-1})\rceil)^2}\bigr),R\Bigr)^2+1\right),
\end{equation}
then $\tmix(\eps,B_R)\leq\bigl\lceil\log(2\eps^{-1})\bigr\rceil$.
\end{theorem}

According to \eqref{eq:DPMCcoupmixt}, noising times sufficient to guarantee geometric mixing to accuracy $\eps$ uniformly from start in $B_R$ can be chosen of order
\begin{equation}\label{eq:suffmix}
    \mathcal O\left(\log\max\Bigl(\Var(\mu),\diam\bigl(\mu,\tfrac{\eps^2}{(4e\lceil\log(2\eps^{-1})\rceil)^2}\bigr),R,e\Bigr)\right).
\end{equation}

\begin{example}\label{ex:tails}
\Cref{thm:DPMCcoupmix,cor:oneDPMC} guarantee geometric, respectively one-step, mixing to a given accuracy for suitably large noising times in general target distributions.
To illustrate the results, we specify their assertions for some important classes of target distributions distinguished based on their tails.
We say the target distribution $\mu$ has tail profile $f$, where $f:[0,\infty)\to[0,\infty)$ is bounded and non-increasing, if
\begin{equation}\label{eq:tail}
    \mu\left(B_{\frac12\diam(\mu,\frac12)+r}^\comp\right)\ \leq\ f(r)\quad\text{for all $r\geq0$.}
\end{equation}
With the generalized inverse $f^{-1}(\delta):=\inf\{r\geq0:f(r)\leq\delta\}$, it then holds
\begin{equation}\label{eq:diamtail}
    \diam(\mu,\delta)\ \leq\ \diam(\mu,\tfrac12)\ +\ 2\,f^{-1}(\delta)\quad\text{for all $\delta>0$.}
\end{equation}
This formulation divides the centered partial diameter into the centered diameter of the bulk, containing one half of the mass, and the width of the tails.
Further, the variance of $\mu$ can be controlled in terms of these notions.
Indeed, if $\int_{\R^d}x\,\mu(dx)=0$, by Cavalieri’s principle,
\begin{align*}
    \Var(\mu)\ &=\ \int_{\R^d}|x|^2\,\mu(dx)\ =\ \int_0^\infty2r\mu(B_{r}^\comp)\,dr\\
    &\leq\ \bigl(\tfrac12\diam(\mu,\tfrac12)\bigr)^2+\int_0^\infty\bigl(2r+\diam(\mu,\tfrac12)\bigr)f(r)\,dr,
\end{align*}
where the last step splits the range of integration and changes variables to insert \eqref{eq:tail}.

Inserting \eqref{eq:diamtail} and the previous display into \eqref{eq:suffmixone} and \eqref{eq:suffmix} shows that noising times sufficient for \emph{geometric} mixing to accuracy $\eps$ uniformly from start in $B_R$ can be chosen of order
\begin{equation}\label{eq:suffmixtails}
    \mathcal O\left(\log\max\Bigl(\diam(\mu,\tfrac12),\int_0^\infty rf(r)\,dr,f^{-1}\bigl(\tfrac{\eps^2}{(4e\lceil\log(2\eps^{-1})\rceil)^2}\bigr),R,e\Bigr)\right),
\end{equation}
while, for achieving the accuracy \emph{in one transition}, sufficient noising times can be chosen of order
\begin{equation}\label{eq:suffmixtailsone}
    \mathcal O\Bigl(\log\max\bigl(\diam(\mu,\tfrac12),f^{-1}(\tfrac\eps3),\eps^{-1},R,e\bigr)\Bigr).
\end{equation}

We specify these scaling in three classes of tails:
\begin{itemize}
\item[(i)]
The target has \emph{sub-exponential tails} if \eqref{eq:tail} holds with $f(r)=Ce^{-cr}$ for $c>0$ and a constant $C>0$.
For such tails, noising times sufficient for geometric mixing to accuracy $\eps$ uniformly from start in $B_R$ can be chosen of order
\begin{equation}\label{eq:texptails}
    \mathcal O\Bigl(\log\max\bigl(\diam(\mu,\tfrac12),c^{-1},\log\eps^{-1},R,e\bigr)\Bigr),
\end{equation}
while noising times sufficient to achieve an $\eps$-accurate sample in one step can be chosen of order
\begin{equation*}
    \mathcal O\Bigl(\log\max\bigl(\diam(\mu,\tfrac12),c^{-1},\eps^{-1},R,e\bigr)\Bigr).
\end{equation*}

These scalings differ only in their dependence on the inverse accuracy.
In particular, noising times scaling \emph{logarithmically} in $\eps^{-1}$ suffice to guarantee an $\eps$-accurate sample in one step, while a \emph{double logarithmic} scaling suffices to guarantee geometric mixing to said accuracy.

Assuming lighter sub-Gaussian tails, meaning that \eqref{eq:tail} holds with $f(r)=Ce^{-cr^2}$ for $c>0$ and a constant $C>0$, does not improve the scalings by more than multiplicative constants due to the logarithm.

\item[(ii)] 
The target has \emph{stretched exponential tails} if \eqref{eq:tail} holds with $f(r)=Ce^{-cr^a}$ for $c>0$ and constants $a\in(0,1)$, $C>0$.
For such tails, noising times sufficient for geometric mixing to accuracy $\eps$ uniformly from start in $B_R$ can be chosen of order
\begin{equation*}
    \mathcal O\Bigl(\log\max\bigl(\diam(\mu,\tfrac12),c^{-1/a},\Gamma(1+2/a),\log^{1/a}\eps^{-1},R,e\bigr)\Bigr),
\end{equation*}
while noising times sufficient to achieve an $\eps$-accurate sample in one step can be chosen of order
\begin{equation*}
    \mathcal O\Bigl(\log\max\bigl(\diam(\mu,\tfrac12),c^{-1/a},\eps^{-1},R,e\bigr)\Bigr).
\end{equation*}

\item[(iii)]
The target has \emph{polynomial tails} if \eqref{eq:tail} holds with $f(r)=C(1+cr)^{-p}$ for $c>0$ and constants $p>0$, $C>0$.
For such tails, the variance is only guaranteed to be finite if $p>2$.
In this case, noising times sufficient for geometric mixing to accuracy $\eps$ uniformly from start in $B_R$ can be chosen of order
\begin{equation*}
    \mathcal O\Bigl(\log\max\bigl(\diam(\mu,\tfrac12),c^{-1},(p-2)^{-1},\eps^{-1/p},R,e\bigr)\Bigr).
\end{equation*}
As the one-step result \Cref{cor:oneDPMC} does not require finite variance, noising times sufficient to achieve an $\eps$-accurate sample in one step can be chosen of order
\begin{equation*}
    \mathcal O\Bigl(\log\max\bigl(\diam(\mu,\tfrac12),c^{-1},\eps^{-(1+1/p)},R,e\bigr)\Bigr)\quad\text{for all $p>0$.}
\end{equation*}

Notably, the scaling difference between noising times sufficient to guarantee an $\eps$-accurate sample in one step compared to geometric mixing to said accuracy ($\log\eps^{-1}$ vs. $\log\log\eps^{-1}$), observed in the presence of sub-exponential tails, disappears when proceeding to heavier polynomial tails.
\end{itemize}

The tails above arise, for instance, under various functional inequalities:
If the log-Sobolev inequality \eqref{eq:LSI} holds, the target distribution has sub-Gaussian tails with $c\propto C_{LS}^{-1}$, see \cite{Ledoux99}.
More generally, if the Poincar\'e inequality \eqref{eq:PI} holds, the distribution has sub-exponential tails with $c\propto C_P^{-1/2}$, see \cite{Bobkov97}.
Similarly, weak or weighted Poincar\'e inequalities accommodate the heavier stretched exponential and polynomial tails, see \cite{Rockner01,Barthe05,Bobkov09}.
Note, however, that the reverse implications are false as the functional inequalities encode structure beyond the tail profile of the distribution.
\end{example}

\begin{proof}[Proof of \Cref{thm:DPMCcoupmix}]
We first show \eqref{eq:DPMCcoupmix}, which implies the subsequently stated mixing time bound.
Let $R>0$ and $n\geq0$.
Taking the supremum over all $x\in B_R$ in the assertion of \Cref{thm:mixDPMC} with $D=B_R$ and $\nu=\delta_x$ establishes
\[ \sup_{x\in B_R}\TV\bigl(\pi_{(t)}^n(x,\cdot),\mu\bigr)\ \leq\ \left(\frac{R}{\sqrt{e^t-1}}\right)^n+\sup_{x\in B_R}\mathbb P_x(T<n)+\mathbb P_\mu(T<n). \]
Thus, to conclude \eqref{eq:DPMCcoupmix}, it suffices to show
\begin{equation}\label{eq:DPMCcoupmixTS}
    \sup_{x\in B_R}\mathbb P_x(T<n)+\mathbb P_\mu(T<n)\ \leq\ 2n\,\mu(B_R^\comp)^{1/2}\exp\left(\frac{\Var(\mu)+R^2}{e^t-1}\right),
\end{equation}
which we do via \Cref{prop:stab}.
Specifically, we control the probability $\pi_{(t)}(x,B_R^\comp)$ of exiting $B_R$ in one step uniformly in $x\in B_R$.
Therefore, let $f(x):=\Pb_t(x,B_R^\comp)$.
Then, it holds for all $x\in\R^d$ that
\[ \pi_{(t)}(x,B_R^\comp)\ =\ \Pf_tf(x)\ =\ \int_{\R^d}f(y)\frac{\Pf_t(x,y)}{\mu_t(y)}\,\mu_t(dy), \]
where $\mu_t(y)$ and $\Pf_t(x,y)$ denote the Lebesgue densities of $\mu_t$ and $\Pf_t(x,\cdot)$, respectively.
Applying the Cauchy--Schwarz inequality to the previous display yields
\begin{equation}\label{eq:piexit}
    \pi_{(t)}(x,B_R^\comp)\ \leq\ \|f\|_{L^2(\mu_t)}\left\|\frac{\Pf_t(x,y)}{\mu_t(y)}\right\|_{L^2(\mu_t)}\ \leq\ \mu(B_R^\comp)^{1/2}\Bigl(\chi^2\bigl(\Pf_t(x,\cdot)\big|\mu_t\bigr)+1\Bigr)^{1/2},
\end{equation}
where we used that, since $\Pb_t(x,B_R^\comp)\in[0,1]$, the first factor satisfies
\[ \|f\|_{L^2(\mu_t)}^2\ =\ \int_{\R^d}\mu_t(dx)\Pb_t(x,B_R^\comp)^2\ \leq\ \mu_t\Pb_t(B_R^\comp)\ =\ \mu(B_R^\comp), \]
while the second factor can be written as
\[ \left\|\frac{\Pf_t(x,y)}{\mu_t(y)}\right\|_{L^2(\mu_t)}^2-1\ =\ \left\|\frac{\Pf_t(x,y)}{\mu_t(y)}-1\right\|_{L^2(\mu_t)}^2\ =\ \chi^2\bigl(\Pf_t(x,\cdot)\big|\mu_t\bigr). \]
The following lemma estimates the $\chi^2$-divergence in \eqref{eq:piexit}.

\begin{lemma}\label{lem:chi2}
Let $t>0$ and $\mu\in\calP(\R^d)$ such that $\int_{\R^d}x\,\mu(dx)=0$.
The noising kernel $\Pf_t$, see \eqref{eq:Pf}, satisfies
\[ \chi^2\bigl(\Pf_t(x,\cdot)\big|\mu_t\bigr)\ \leq\ \exp\left(\frac{\frac12\Var(\mu)+|x|^2}{e^t-1}\right)-1\quad\text{for all $x\in\R^d$,} \]
where $\mu_t=\mu\Pf_t$ and $\Var(\mu)=\int_{\R^d}|x|^2\mu(dx)$.
\end{lemma}

\noindent
Its proof is given subsequent to the current proof of \Cref{thm:DPMCcoupmix}.
Inserting the lemma's assertion into \eqref{eq:piexit} shows
\[ \sup_{x\in B_R}\pi_{(t)}(x,B_R^\comp)\ \leq\ \mu(B_R^\comp)^{1/2}\exp\left(\frac{\Var(\mu)+R^2}{e^t-1}\right). \]
Inserting this bound into the exit probability estimate of \Cref{prop:stab} from start in $\nu=\delta_x$ and taking the supremum over all $x\in B_R$ yields
\begin{equation}\label{eq:BRexitprob}
    \sup_{x\in B_R}\mathbb P_x(T<n)\ \leq\ n\,\mu(B_R^\comp)^{1/2}\exp\left(\frac{\Var(\mu)+R^2}{e^t-1}\right).
\end{equation}
From start in $\mu$, \Cref{prop:stab} asserts
\[ \mathbb P_\mu(T<n)\ \leq\ n\,\mu(B_R^\comp), \]
which is bounded above by the right hand side of \eqref{eq:BRexitprob}.
Therefore,
\[ \sup_{x\in B_R}\mathbb P_x(T<n)+\mathbb P_\mu(T<n)\ \leq\ 2n\,\mu(B_R^\comp)^{1/2}\exp\left(\frac{\Var(\mu)+R^2}{e^t-1}\right), \]
which shows \eqref{eq:DPMCcoupmixTS} and hence finishes the proof of \eqref{eq:DPMCcoupmix}.

We now turn to the asserted mixing time bound.
Let $\eps>0$ and $R>0$.
Set $n=\lceil\log(2\eps^{-1})\rceil$ and $R'=\max\bigl(\frac12\diam(\mu,(\frac{\eps}{4en})^2),R\bigr)$.
By \eqref{eq:DPMCcoupmix} applied to the ball of radius $R'\geq R$,
\[ \sup_{x\in B_R}\TV\bigl(\pi_{(t)}^n(x,\cdot),\mu\bigr)\ \leq\ \left(\frac{R'}{\sqrt{e^t-1}}\right)^n+2n\,\mu(B_{R'}^\comp)^{1/2}\exp\left(\frac{\Var(\mu)+R'^2}{e^t-1}\right). \]
To control the two terms on the right hand side, let $t\geq\log\bigl(\Var(\mu)+(eR')^2+1\bigr)$, that is, as in \eqref{eq:DPMCcoupmixt}.
Then, as $t\geq\log\bigl((eR')^2+1\bigr)$, we have
\[ \left(\frac{R'}{\sqrt{e^t-1}}\right)^n\ \leq\ e^{-n}\ \leq\ \frac{\eps}{2}. \]
Further, by definition of the centered partial diameter, see \eqref{eq:diam}, $\mu(B_{R'}^\comp)\leq(\frac{\eps}{4en})^2$.
As $t\geq\log\bigl(\Var(\mu)+R'^2+1\bigr)$, it holds that
\[ 2n\,\mu(B_{R'}^\comp)^{1/2}\exp\left(\frac{\Var(\mu)+R'^2}{e^t-1}\right)\ \leq\ 2en\,\mu(B_{R'}^\comp)^{1/2}\ \leq\ \frac\eps2. \]
The previous three displays together imply
\[ \sup_{x\in B_R}\TV\bigl(\pi_{(t)}^n(x,\cdot),\mu\bigr)\ \leq\ \eps, \]
which establishes $\tmix(\eps,B_R)\leq n=\lceil\log(2\eps^{-1})\rceil$, closing the proof.
\end{proof}

Let us now turn to the proof of \Cref{lem:chi2}.

\begin{proof}[Proof of \Cref{lem:chi2}]
Let $x\in\R^d$.
The noising kernel is
\[ \Pf_t(x,\cdot)\ =\ \N\bigl(e^{-t/2}x,(1-e^{-t})I_d\bigr) \]
with Lebesgue density denoted by $\Pf_t(x,y)$.
The definition of the $\chi^2$-divergence shows
\begin{equation}\label{eq:chi2warm}
    \chi^2\bigl(\Pf_t(x,\cdot)\big|\mu_t\bigr)\ =\ \int_{\R^d}\frac{\Pf_t(x,y)^2}{\mu_t(y)}\,dy-1.
\end{equation}
By Jensen's inequality, the Lebesgue density of $\mu_t=\mu\Pf_t$ satisfies, for all $y\in\R^d$,
\begin{align*}
\mu_t(y)
\ &=\ \int_{\R^d}\mu(dx)\Pf_t(x,y)
\ =\ \bigl(2\pi(1-e^{-t})\bigr)^{-d/2}\int_{\R^d}\exp\left(-\frac{|y-e^{-t/2}x|^2}{2(1-e^{-t})}\right)\mu(dx) \\
&\geq\ \bigl(2\pi(1-e^{-t})\bigr)^{-d/2}\exp\left(-\int_{\R^d}\frac{|y-e^{-t/2}x|^2}{2(1-e^{-t})}\,\mu(dx)\right) \\
&=\ \exp\left(-\frac{\Var(\mu)}{2(e^t-1)}\right)\Pf_t(0,y),
\end{align*}
where, in the last step, we used
\[ \int_{\R^d}|y-e^{-t/2}x|^2\,\mu(dx)\ =\ e^{-t}\Var(\mu)+|y|^2. \]
Inserting this lower bound on $\mu_t(y)$ into \eqref{eq:chi2warm} yields
\begin{align*}
\chi^2\bigl(\Pf_t(x,\cdot)\big|\mu_t\bigr)
\ &\leq\ \exp\left(\frac{\Var(\mu)}{2(e^t-1)}\right)\int_{\R^d}\frac{\Pf_t(x,y)^2}{\Pf_t(0,y)}\,dy-1\\
&=\ \exp\left(\frac{\frac12\Var(\mu)+|x|^2}{e^t-1}\right)-1,
\end{align*}
where we used
\begin{align*}
\int_{\R^d}\frac{\Pf_t(x,y)^2}{\Pf_t(0,y)}\,dy
\ &=\ \int_{\R^d}\bigl(2\pi(1-e^{-t})\bigr)^{-d/2}\exp\left(-\frac{|y-2e^{-t/2}x|^2-2e^{-t}|x|^2}{2(1-e^{-t})}\right)dy \\
&=\ \exp\left(\frac{|x|^2}{e^t-1}\right).
\end{align*}
\end{proof}

\subsection{Mixing under a Spectral Gap}\label{sec:gapmix}

In the following theorem, we observe the ability of ideal Diffusion Path Monte Carlo to generate a warm initial distribution in one transition, see \eqref{eq:DPMCwarm}, even when initialized in point masses, which themselves have infinite $\chi^2$-divergence to equilibrium.
This enables a spectral gap based mixing time bound without restricting to warm initial distributions.

\begin{theorem}[Mixing under Spectral Gap]\label{thm:mixSG}
Let $\pi_{(t)}$ be the transition kernel with invariant distribution $\mu\in\calP(\R^d)$ of ideal Diffusion Path Monte Carlo with noising time $t>0$, or equivalently of the Proximal Sampler with step size $h=e^t-1$.
Assume $\int_{\R^d}x\,\mu(dx)=0$.
Then,
\begin{equation}\label{eq:DPMCwarm}
    \chi^2\bigl(\pi_{(t)}(x,\cdot)\big|\mu\bigr)\ \leq\ \exp\left(\frac{\Var(\mu)+|x|^2}{e^t-1}\right)-1\quad\text{for all $x\in\R^d$.}
\end{equation}
In particular, for all $\eps>0$ and $R>0$,
\begin{equation}\label{eq:mixSG}
    \tmix(\eps,B_R)\ \leq\ \left\lceil\Gap\bigl(\pi_{(t)}\bigr)^{-1}\left(\frac{\Var(\mu)+R^2}{e^t-1}+\log\eps^{-1}\right)\right\rceil+1.
\end{equation}
\end{theorem}

Similarly, the relative entropy from $\pi_{(t)}(x,\cdot)$ to equilibrium can be controlled uniformly in $x\in B_R$, which, via Pinsker's inequality, establishes mixing time bounds for general initial distributions based on contraction to equilibrium in relative entropy.

\begin{example}
To compare the bounds on the mixing time to accuracy $\eps>0$ from start in $B_R$ provided by the Dobrushin contraction based \Cref{thm:DPMCcoupmix} and the spectral gap based \Cref{thm:mixSG}, consider a centered mixture $\mu=\sum_{i=1}^Nw_i\mu_i$ of $N\geq1$ components $\mu_i\in\calP(\R^d)$, each satisfying the Poincar\'e inequality \eqref{eq:PI}, with non-negative weights $\sum_{i=1}^Nw_i=1$.

In this setting, \cite{sifanDP} lower bound the spectral gap.
Inserting their estimate, see \eqref{eq:gapPImix}, into the assertion \eqref{eq:mixSG} of \Cref{thm:mixSG} yields
\[ \tmix(\eps,B_R)\ \leq\ \left\lceil\left(1+\frac{C_P}{e^t-1}\right)\exp\left(\frac{W^2}{4(e^t-1)}\right)\left(\frac{C_Pd+W^2+R^2}{e^t-1}+\log\eps^{-1}\right)\right\rceil+1, \]
where $W:=\max_{1\leq i,j\leq N}\W_2(\mu_i,\mu_j)$ and we used that $\Var(\mu)\leq C_Pd+\frac12W^2$.
In particular, there exist noising times
\begin{equation}\label{eq:mixtSG}
    t\ =\ \mathcal O\Bigl(\log\max\bigl(C_Pd,W,R,e\bigr)\Bigr)
\end{equation}
such that $\tmix(\eps,B_R)=\mathcal O(\lceil\log\eps^{-1}\rceil)$.

The bound on the variance follow from
\[ \Var(\mu)\ \leq\ \sum_{i=1}^Nw_i\Var(\mu_i) + \frac12\sum_{i,j=1}^Nw_iw_j|m_i-m_j|^2, \]
where $m_i=\int_{\R^d}x\,\mu_i(dx)$, together with (i) the fact that, with the $k$-th component map $f_i^k(x):=x^k-m_i^k$, we have $\Var(\mu_i)=\sum_{k=1}^d\Var_{\mu_i}(f_i^k)\leq C_Pd$ by the Poincar\'e inequality \eqref{eq:PI} for $\mu_i$, and (ii) $|m_i-m_j|^2\leq\E|X_i-X_j|^2$ for any coupling $(X_i,X_j)$ of $\mu_i$ and $\mu_j$ so that taking the infimum over all such coupling shows $|m_i-m_j|^2\leq\W_2(\mu_i,\mu_j)^2\leq W^2$.

Further, since the mixture $\mu$ has sub-exponential tails with $c\propto C_P^{-1/2}$, see \cite{Chafai10}, we can invoke the implications of \Cref{thm:DPMCcoupmix} derived in \Cref{ex:tails}.
Specifically, using that $\diam(\mu,\tfrac12)\leq 2\sqrt{2\Var(\mu)}$ by Chebyshev's inequality together with the variance bound above, \eqref{eq:texptails} asserts the existence of noising times
\begin{equation}\label{eq:mixtDob}
    t\ =\ \mathcal O\Bigl(\log\max\bigl(C_Pd, W,\log\eps^{-1},R,e\bigr)\Bigr)
\end{equation}
such that $\tmix(\eps,B_R)=\mathcal O(\lceil\log\eps^{-1}\rceil)$.

Comparing \eqref{eq:mixtSG} and \eqref{eq:mixtDob} shows the spectral gap and Dobrushin contraction based approaches of \Cref{thm:mixSG,thm:DPMCcoupmix} to produce consistent results in the setting of mixture targets consisting of components satisfying a Poincar\'e inequality.
The noising time scaling obtained from Dobrushin contraction features a mild additional double logarithmic dependence on the inverse accuracy, which does not appear in the corresponding scaling based on spectral gaps.
This is expected since, while a spectral gap implies geometric convergence to equilibrium in $\chi^2$-divergence, the Dobrushin contraction approach only asserts local geometric mixing.
To ensure sufficient stability to achieve a given accuracy, the accuracy then enters the noising time.
\end{example}

\begin{proof}[Proof of \Cref{thm:mixSG}]
Let $x\in\R^d$.
Using that the $\chi^2$-divergence is non-increasing under the action of any Markov kernel, also known as the data processing inequality, we see
\[ \chi^2\bigl(\pi_{(t)}(x,\cdot)\big|\mu\bigr)
\ =\ \chi^2\bigl(\Pf_t(x,\cdot)\Pb_t\big|\mu_t\Pb_t\bigr)
\ \leq\ \chi^2\bigl(\Pf_t(x,\cdot)\big|\mu_t\bigr). \]
By \Cref{lem:chi2}, it follows that
\[ \chi^2\bigl(\pi_{(t)}(x,\cdot)\big|\mu\bigr)\ \leq\ \exp\left(\frac{\frac12\Var(\mu)+|x|^2}{e^t-1}\right)-1, \]
showing the first assertion \eqref{eq:DPMCwarm}.
To conclude the mixing time bound, let $\eps>0$ and $R>0$.
By the first assertion, for all $x\in B_R$, $\pi_{(t)}(x,\cdot)$ is $\beta$-warm with
\[ \beta\ =\ \exp\left(\frac{\Var(\mu)+R^2}{e^t-1}\right). \]
The mixing time bound from $\beta$-warm initial distributions \eqref{eq:mixwarm} then yields
\[ \tmix\bigl(\eps,\pi_{(t)}(x,\cdot)\bigr)\ \leq\ \left\lceil\Gap\bigl(\pi_{(t)}\bigr)^{-1}\left(\frac{\Var(\mu)+R^2}{e^t-1}+\log\eps^{-1}\right)\right\rceil. \]
Hence
\begin{align*}
\tmix(\eps,B_R)
\ &=\ \sup_{x\in B_R}\tmix(\eps,\delta_x)
\ \leq\ \sup_{x\in B_R}\tmix\bigl(\eps,\pi_{(t)}(x,\cdot)\bigr) + 1\\
&\leq\ \left\lceil\Gap\bigl(\pi_{(t)}\bigr)^{-1}\left(\frac{\Var(\mu)+R^2}{e^t-1}+\log\eps^{-1}\right)\right\rceil+1,
\end{align*}
finishing the proof.
\end{proof}

\section{Metropolis-adjusted Diffusion Path Monte Carlo}\label{sec:MDPMC}

After presenting and studying the ideal Diffusion Path Monte Carlo kernel in the previous two sections, we now turn to its implementable counterpart---Metropolis-adjusted Diffusion Path Monte Carlo.
In the present section, we present the method recently introduced in \cite{sifanDP,Hill26}.
On the basis of Metropolis--Hastings' formulation in general state spaces of \cite{Tierney98}, we demonstrate its reversibility in general target probability distributions.
Studying the mixing of the Metropolis-adjusted method is postponed to the subsequent \Cref{sec:MDPMCmix}.

\subsection{Implementing the Noising-Denoising Transition}

In practice, ideal Diffusion Path Monte Carlo is not implementable as exact draws from the denoising kernel $\Pb_t$ are not accessible.
However, by discretizing the denoising equation \eqref{eq:Xb} in a computationally feasible way, one may obtain approximate samples to resort to.
Therein, two aspects need to be resolved:
On the one hand, the continuous time stochastic differential equation needs to be discretized, for which we employ the Euler--Maruyama scheme.
Denoting its step size by $h>0$, we assume $t\in h\mathbb N$.
While $h$ previously denoted the step size of the Proximal Sampler, from now on, $h$ always depicts the Euler--Maruyama step size.
On the other hand, and more importantly, the score function $\nabla\log\mu_s(x):(0,t)\times\R^d\to\R^d$, appearing in the denoising equation, is generally not known exactly.
Instead, it has to be replaced by an \emph{approximate score} $\mathfrak s_s(x):\{kh\}_{1\leq k\leq t/h}\times\R^d\to\R^d$.
In practice, such approximate scores are learned.
For instance, \cite{sifanDP} suggests minimizing a relative entropy objective on diffusion path space using the unnormalized target density, whereas \cite{Hill26} proposes adding noise to samples from short MCMC runs and training a model to remove that noise.
In our theoretical considerations, the approximate score remains a black box.

While the Ornstein--Uhlenbeck noising process could be implemented exactly, we choose to discretize it as well due to reasons becoming apparent later.
The Euler--Maruyama scheme using the approximate score yields the following surrogates for the noising and denoising processes in \eqref{eq:OU} and \eqref{eq:Xb}, respectively.
On the one hand,
\[ \Yf_{k+1}\ =\ (1-\tfrac12h)\Yf_{k}\ +\ h^{1/2}\vec Z_k\quad\text{for $k\geq0$,} \]
where $(\vec Z_k)_{k\geq0}$ is a sequence of independent $\N(0,I_d)$-distributed random variables, independent of $\Yf_0$.
On the other hand,
\[ \Yb_{k+1}\ =\ (1+\tfrac12h)\Yb_{k}\ +\ h\,\mathfrak s_{t-kh}\bigl(\Yb_{k}\bigr)\ +\ h^{1/2}\cev Z_k\quad\text{for $0\leq k<t/h$,} \]
where $\cev Z_0,\dots,\cev Z_{t/h-1}\sim\N(0,I_d)$ are mutually independent and independent of $\Yb_0$.
For $0\leq k<t/h$, the corresponding one-step transition kernels on $\R^d$ are
\begin{equation}\label{eq:Qf}
    \Qf_{k+1}(y,\cdot)\ :=\ \Law\bigl(\Yf_{k+1}\big|\Yf_{k}=y\bigr)\ =\ \N\bigl((1-\tfrac12h)y,hI_d\bigr)
\end{equation}
and
\begin{equation}\label{eq:Qb}
    \Qb_{k+1}(y,\cdot)\ :=\ \Law\bigl(\Yb_{k+1}\big|\Yb_{k}=y\bigr)\ =\ \N\bigl((1+\tfrac12h)y+h\,\mathfrak s_{t-kh}(y),hI_d\bigr).
\end{equation}  
Note that, while $\Qf_{k+1}$ depend on the step size $h$ only, $\Qb_{k+1}$ additionally depend on the noising time $t$ and the score approximation $\mathfrak s$.

To obtain a computational approximation of the ideal Diffusion Path Monte Carlo transition kernel $\pi_{(t)}=\Pf_t\Pb_t$, which concatenate ideal noising and denoising
\[ \Pf_t(x,\cdot)\ =\ \Law\bigl(\Xf_t\big|\Xf_0=x\bigr)\quad\text{and}\quad\Pb_t(x,\cdot)\ =\ \Law\bigl(\Xb^{(t)}_t\big|\Xb^{(t)}_0=x\bigr), \]
we concatenate the approximate noising and denoising
\begin{equation}\label{eq:Qft}
    \Qf_{(t)}(y,\cdot)\ :=\ \Law\bigl(\Yf_{t/h}\big|\Yf_0=y\bigr)\quad\text{and}\quad\Qb_{(t)}(y,\cdot)\ :=\ \Law\bigl(\Yb_{t/h}\big|\Yb_0=y\bigr)
\end{equation}
to $Q_{(t)}:=\Qf_{(t)}\Qb_{(t)}$.

\subsection{The Metropolis--Hastings Method}\label{sec:MH}

Due to numerical discretization and use of an approximate score, the transition kernel $Q_{(t)}$ generally does not leave the target $\mu$ invariant.
The Metropolis--Hastings method \cite{Metropolis53,Hastings70}, a flexible framework to construct Markov transitions that satisfy detailed balance with respect to a given target $\mu$ and hence leave it invariant, allows to correct said bias.

The following general state space formulation is due to Tierney, see \cite{Tierney98}.
Let $\mu$ be a target probability measure and $q$ a proposal Markov transition kernel on a measurable space $(\mathcal E,\mathfrak E)$.
Further, for $x,y\in\mathcal E$, define the acceptance probability
\begin{equation}\label{eq:alpha_gen}
\alpha(x,y)\ =\ \begin{cases}
\min\bigl(\rho(y,x),1\bigr) & \text{if $(x,y)\in\mathcal R$,}\\
0                           & \text{otherwise,}
\end{cases}
\end{equation}
where $\mathcal R$ is the (up to null-sets with respect to both following measures) unique, symmetric, measurable subset of $\mathcal E\times\mathcal E$ such that $\mu(dy)q(y,dx)$ and $\mu(dx)q(x,dy)$ are mutually absolutely continuous in $\mathcal R$ and mutually singular in its complement, and $\rho(y,x)$ denotes a version of the relative density of $\mu(dy)q(y,dx)\big|_{\mathcal R}$ with respect to $\mu(dx)q(x,dy)\big|_{\mathcal R}$ such that $0<\rho(y,x)<\infty$ and $\rho(y,x)^{-1}=\rho(x,y)$ for all $x,y\in\mathcal E$.

With these notions, a transition $X\sim\pi_{\MH}(x,\cdot)$ of Metropolis--Hastings proceeds as follows:
First, a transition to $Y\sim q(x,\cdot)$ is proposed, which is accepted with probability $\alpha(x,Y)$, in which case $X=Y$, and otherwise rejected, in which case the chain remains in $x$, that is, $X=x$.
The transition kernel takes the form
\[ \pi_{\MH}(x,dy)\ =\ \alpha(x,y)\,q(x,dy)\ +\ \int_{\mathcal E}\bigl(1-\alpha(x,z)\bigr)\,q(x,dz)\,\delta_x(dy). \]
As shown in \cite{Tierney98}, the transition kernel $\pi_{\MH}$ satisfies detailed balance with respect to the target $\mu$, i.e.,
\[ \mu(dx)\,\pi_{\MH}(x,dy)\ =\ \mu(dy)\,\pi_{\MH}(y,dx) \]
as measures on $\mathfrak E\otimes\mathfrak E$ and hence leaves $\mu$ invariant.

In principle, the approximate noising-denoising transition kernel $Q_{(t)}$ can directly be Metropolis-adjusted by regarding it as proposal kernel $q$.
However, the resulting acceptance probability is not analytically tractable, and the Metropolis correction hence not computationally feasible.
To see this, for simplicity, assume mutual absolute continuity $\mu\sim\lambda$ between $\mu$ and Lebesgue measure $\lambda$ on $\R^d$.
As by construction, $Q_{(t)}(x,\cdot)\sim\lambda$ for all $x\in\R^d$, it then holds that $\mu(dy)Q_{(t)}(y,dx)\sim\mu(dx)Q_{(t)}(x,dy)$, that is, one may set $\mathcal R=\R^d\times\R^d$.
With the respective Lebesgue densities, the Metropolis acceptance probability reads
\begin{equation}\label{eq:alphaQt}
    \alpha(x,y)\ =\
\min\left(\frac{\mu(y)Q_{(t)}(y,x)}{\mu(x)Q_{(t)}(x,y)},1\right).
\end{equation}
The reason for the acceptance probability not being analytically tractable lies in $Q_{(t)}(y,\cdot)$, arising as endpoint law $\Law\bigl(\Yb_{t/h}\bigr)$ of the noising path $\bigl(\Yf_0,\Yf_1,\dots,\Yf_{t/h}\bigr)$ started in $\Yf_0=y$ followed by the denoising path $\bigl(\Yb_0,\Yb_1,\dots,\Yb_{t/h}\bigr)$ started in the noising path's endpoint $\Yb_0=\Yf_{t/h}$.
From the one-step transition laws \eqref{eq:Qf} and \eqref{eq:Qb}, the law of the entire noising-denoising path $\hat y=\bigl(\hat y_0,\hat y_1,\dots,\hat y_{t/h},\dots,\hat y_{2t/h}\bigr)$ in $\R^{d(2t/h+1)}$, given $\Yf_0=y$, takes the form
\begin{equation}\label{eq:Q}
    Q(y,d\hat y)\ :=\ \delta_y(d\hat y_0)\left(\prod_{k=0}^{t/h-1}\Qf_{k+1}(\hat y_k,d\hat y_{k+1})\right)\left(\prod_{k=0}^{t/h-1}\Qb_{k+1}(\hat y_{t/h+k},d\hat y_{t/h+k+1})\right).
\end{equation}
The endpoint law emerges from integrating out the intermediate states $\hat y_0,\dots,\hat y_{2t/h-1}$ along the path initialized in $y$, that is,
\begin{equation}\label{eq:Qt}
    Q_{(t)}(y,\cdot)\ =\ \int_{\R^d}\dots\int_{\R^d}Q\bigl(y,d\hat y_0,\dots,d\hat y_{2t/h-1},\,\cdot\,\bigr).
\end{equation}
However, the integrals with respect to the denoising kernels $\Qb_{k+1}$, which involve the approximate score function $\mathfrak s$, are generally not analytically solvable.
This bars us from exactly evaluating the Metropolis acceptance probability and hence implementing the Metropolis--Hastings transitions with proposal $Q_{(t)}$.

\subsection{Metropolis-adjusted Diffusion Path Monte Carlo}

To obtain a computationally feasible Metropolis-adjustment of $Q_{(t)}$, instead of regarding itself as proposal kernel within Metropolis--Hastings, \cite{sifanDP,Hill26} interpret the dynamics as transitions on the space $\R^{d(N+1)}$ of noising-denoising paths $\hat y=\bigl(\hat y_0,\hat y_1,\dots,\hat y_{t/h},\dots,\hat y_N\bigr)$, where we set $N=2t/h$.
Then, the path law $Q(y,\cdot)$ given the initial state $y$, see \eqref{eq:Q}, can be understood as Markov transition kernel $\hat Q(\hat y,\cdot):=Q(\hat y_0,\cdot)$ on path space that resamples a noising-denoising path with the same initial state as $\hat y$.
Combined with the \emph{path reversal involution}
\[ R\hat y\ =\ R(\hat y_0,\hat y_1,\dots,\hat y_N)\ :=\ (\hat y_N,\hat y_{N-1},\dots,\hat y_0), \]
the transition $Q_{(t)}(y,\cdot)$ can be realized as pushforward $(p\circ R)_\#Q(y,\cdot)$, where
\begin{equation}\label{eq:p}
p:\R^{d(N+1)}\to\R^d,\quad p(\hat y) = \hat y_0
\end{equation}
denotes the \emph{projection onto the initial state}.
For a Markov transition kernel $\hat\pi(\hat y,\cdot)$ on path space that only depends on $\hat y$ through $p(\hat y)$, the \emph{projected Markov kernel} on $\R^d$ is defined by
\begin{equation}\label{eq:proj}
    (p_\#\hat\pi)(y,B)\ :=\ \hat\pi\bigl(\hat y,p^{-1}(B)\bigr)\quad\text{with any $\hat y\in p^{-1}(y)$}
\end{equation}
for all $y\in\R^d$ and Borel $B\subseteq\R^d$.
Writing $\hat R(\hat y,\cdot):=\delta_{R\hat y}$ for the transition kernel induced by path-reversal, we have
\[ Q_{(t)}\ =\ p_\#(\hat Q\hat R)\ =\ p_\#(\hat Q\hat R\hat Q), \]
where the second equality uses that an additional transition along $\hat Q$ leaves the initial state invariant.
The point in regarding $Q_{(t)}$ as projection of $\hat Q\hat R\hat Q$ on path space lies in the fact that, for a projected transition kernel $p_\#\hat\pi$ on $\R^d$ to be reversible with respect to $\mu$, it suffices for $\hat\pi$ to be reversible with respect to the lifted target
\[ \hat\mu\ :=\ \mu Q\ =\ \int_{\R^d}\mu(dy)\,Q(y,\cdot) \]
on path space, for which $p_\#\hat\mu=\mu$.
Crucially, the deterministic path reversal $\hat R$ can be Metropolis-adjusted with computationally feasible acceptance probabilities to a $\hat\mu$-reversible kernel $\hat R_{\MH}$.
Palindromic combination with $\hat Q$, which is itself $\hat\mu$-reversible, to $\hat Q\hat R_{\MH}\hat Q$ remains reversible with respect to $\hat\mu$.
Therefore, its projection $\Pi_{(t)}:=p_\#(\hat Q\hat R_{\MH}\hat Q)$, defining the transition kernel of \emph{Metropolis-adjusted Diffusion Path Monte Carlo} is reversible with respect to $\mu$, serving as computationally feasible Metropolis-adjustment of $Q_{(t)}$.
The following proposition makes this rigorous for general target probability distributions $\mu$ on $\R^d$.

In order to define the acceptance probabilities for general target distributions, note that, by the Lebesgue decomposition and Radon--Nikodym theorems, see \cite{Rudin87} for reference, there exist unique sub-probability measures $\mu_a$ and $\mu_s$ on $\R^d$ such that $\mu_a\ll\lambda$, $\mu_s\perp\lambda$, and
\[ \mu(dy)\ =\ \mu_a(dy)\ +\ \mu_s(dy)\ =\ \mu_a(y)\,dy\ +\ \mu_s(dy), \]
where we denote by $\mu_a$ both the measure as well as its Lebesgue density.
In particular, there exists a Lebesgue null set $S\subset\R^d$ such that $\mu_s(S^\comp)=0$.
The decomposition of $\mu$ yields the decomposition 
\[ \hat\mu(d\hat y)\ =\ \hat\mu_a(d\hat y)\ +\ \hat\mu_s(d\hat y)\ =\ \hat\mu_a(\hat y)\,d\hat y\ +\ \hat\mu_s(d\hat y) \]
into $\hat\mu_a:=\mu_aQ$, which is absolutely continuous with respect to Lebesgue measure on path space, and $\hat\mu_s:=\mu_sQ$.

\begin{proposition}\label{prop:rev}
Assume the transition laws $\Qf_{k+1}(y,\cdot)$ and $\Qb_{k+1}(y,\cdot)$ in \eqref{eq:Q} to be absolutely continuous with respect to Lebesgue measure for all $0\leq k<t/h$ and $y\in\R^d$.
Then, for any $\mu\in\calP(\R^d)$, the transition kernel on path space $\hat Q\hat R_{\MH}\hat Q$ is reversible with respect to $\hat\mu$, where
\[ \hat R_{\MH}(\hat y,\cdot)\ :=\ \alpha(\hat y,R\hat y)\,\delta_{R\hat y}\ +\ \bigl(1-\alpha(\hat y,R\hat y)\bigr)\,\delta_{\hat y} \]
is the Metropolis-adjusted path reversal with acceptance probability
\begin{equation}\label{eq:alphaprop}
\alpha(\hat y,R\hat y)\ =\
\begin{cases}
\min\left(\frac{\hat\mu_a\circ R(\hat y)}{\hat\mu_a(\hat y)},1\right) & \text{if $\hat\mu_a(\hat y)>0$ and $\hat y_0,\hat y_N\in S^\comp$,} \\
1 & \text{if $\hat\mu_a(\hat y)=0$ and $\hat y_0,\hat y_N\in S^\comp$,} \\
0 & \text{otherwise.}
\end{cases}
\end{equation}
In particular, the projected kernel $p_\#(\hat Q\hat R_{\MH}\hat Q)$ on $\R^d$ is reversible with respect to $\mu$.
\end{proposition}

Note that if the transition laws $\Qf_{k+1}$ and $\Qb_{k+1}$ are not only absolutely continuous but mutually absolutely continuous with respect to Lebesgue measure, $\hat\mu_a(\hat y)>0$ if and only if $\mu_a(\hat y_0)>0$.

We are now ready to formally define Metropolis-adjusted Diffusion Path Monte Carlo.

\begin{definition}\label{def:MDPMC}
The transition kernel of \emph{Metropolis-adjusted Diffusion Path Monte Carlo} with step size $h>0$, noising time $t\in h\mathbb N$, and approximate score $\mathfrak s:\{kh\}_{1\leq k\leq t/h}\times\R^d\to\R^d$ is defined by
\[ \Pi_{(t)}\ :=\ p_\#\bigl(\hat Q\hat R_{\MH}\hat Q\bigr), \]
where $\hat Q(\hat y,\cdot)=Q(\hat y_0,\cdot)$ with $Q$ as in \eqref{eq:Q} consisting of the approximate Ornstein--Uhlenbeck noising and denoising kernels \eqref{eq:Qf} and \eqref{eq:Qb}, and $\hat R_{\MH}$ as in \Cref{prop:rev}.
\end{definition}

Since the approximate Ornstein--Uhlenbeck noising and denoising kernels are absolutely continuous with respect to Lebesgue measure, $\Pi_{(t)}$ is reversible with respect to $\mu$ by the proposition.
Although reversibility remains to hold beyond the Ornstein--Uhlenbeck setting, for simplicity, we restrict our considerations thereto.
Further, while the transition kernel, beyond the noising time $t$, depends also on the step size $h$ and approximate score $\mathfrak s$, we omit them from the notation since our analysis primarily concerns $t$.

Note that, due to the absolute continuity of $\Qf_{k+1}$ and $\Qb_{k+1}$ with respect to Lebesgue measure assumed in the proposition, $\hat\mu(\{\hat y_N\in S\})=0$, that is, the noising-denoising path does almost surely not end in $S$.
Therefore, under this assumption, Metropolis-adjusted Diffusion Path Monte Carlo is not irreducible with respect to target distributions $\mu$ that are not absolutely continuous with respect to Lebesgue measure.
This corresponds to a well-known obstruction in the diffusion model literature: discretized denoising schemes with absolutely continuous transition kernels cannot converge to a target distribution with a nonzero singular component in total variation.
For targets supported on lower-dimensional manifolds, this motivates convergence guarantees in Wasserstein distance instead, see \cite{Debortoli22}.
Regarding paths starting or ending in $S$, reversibility only requires the acceptance probability to vanish if $\hat y_0\in S$ and $\hat y_N\notin S$.
On the $\hat\mu$-null set $\{\hat y_N\in S\}$, one can define it arbitrary without affecting reversibility.

If $\mu(dx)=\mu(x)\,dx$ is absolutely continuous with respect to Lebesgue measure, it holds $\mu_a=\mu$ and $S=\emptyset$ so that the acceptance probability \eqref{eq:alphaprop} simplifies to
\begin{equation}
\alpha(\hat y,R\hat y)\ =\
\begin{cases}
\min\left(\frac{\hat\mu\circ R(\hat y)}{\hat\mu(\hat y)},1\right) & \text{if $\hat\mu(\hat y)>0$,} \\
1 & \text{otherwise.} \\
\end{cases}
\end{equation}
With the convention that $\min(\frac{r}{0},1)=1$ for all $r\geq0$, this recovers the acceptance probability given in \cite{sifanDP}, where reversibility of Metropolis-adjusted Diffusion Path Monte Carlo with respect to $\mu$ is shown via a direct computation.
With $N=2t/h$, the quotient in the minimum reads
\[ \frac{\hat\mu\circ R(\hat y)}{\hat\mu(\hat y)}\ =\ \frac{\mu(\hat y_N)}{\mu(\hat y_0)}\frac{\prod_{k=0}^{\frac N2-1}\Qf_{k+1}\bigl(\hat y_{N-k},\hat y_{N-k-1}\bigr)}{\prod_{k=0}^{\frac N2-1}\Qf_{k+1}\bigl(\hat y_k,\hat y_{k+1}\bigr)}\frac{\prod_{k=0}^{\frac N2-1}\Qb_{k+1}\bigl(\hat y_{\frac N2-k},\hat y_{\frac N2-k-1}\bigr)}{\prod_{k=0}^{\frac N2-1}\Qb_{k+1}\bigl(\hat y_{\frac N2+k},\hat y_{\frac N2+k+1}\bigr)}. \]
Assuming the unnormalized densities of $\mu$ and all $\Qf_{k+1}$ and $\Qb_{k+1}$ to be accessible, which they are for the approximate Ornstein--Uhlenbeck noising and denoising kernels, evaluation of this acceptance probability is straightforward.

\begin{proof}[Proof of \Cref{prop:rev}]
First, note that $\hat Q$ is reversible with respect to $\hat\mu$ since
\[ \hat\mu(d\hat x)\hat Q(\hat x,d\hat y)\ =\ \int_{\R^d}\mu(dz)Q(z,d\hat x)Q(\hat x_0,d\hat y)\ =\ \int_{\R^d}\mu(dz)Q(z,d\hat x)Q(z,d\hat y) \]
is symmetric in $d\hat x$ and $d\hat y$, where we used that the first $Q$ involves $\delta_z(d\hat x_0)$.
Therefore, the palindromic $\hat Q\hat R_{\MH}\hat Q$ is reversible if $\hat R_{\MH}$ is.

The reversibility of $\hat R_{\MH}$ follows directly from the general state space Metropolis--Hastings construction reviewed in \Cref{sec:MH}, specified to deterministic proposals along involutions.
Specifically, consider the proposal kernel $q(\hat y,\cdot)=\delta_{R\hat y}$ on path space and note that $\hat\mu$ is absolutely continuous with respect to $\hat\mu+R_\#\hat\mu$ with density
\[ h(\hat y)\ =\
\begin{cases}
\frac{\hat\mu_a(\hat y)}{\hat\mu_a(\hat y)+\hat\mu_a\circ R(\hat y)} & \text{if $\hat\mu_a(\hat y)>0$ and $\hat y_0,\hat y_N\in S^\comp$,}\\
1 & \text{if $\hat y_0\in S$,}\\
0 & \text{otherwise.}
\end{cases} \]
As detailed in \cite{Tierney98}, with
\[ A\ =\ \bigl\{\hat y\in\R^{d(N+1)}:\hat\mu_a(\hat y),\hat\mu_a\circ R(\hat y)>0\text{ and }\hat y_0,\hat y_N\in S^\comp\bigr\}, \]
it then holds $\mathcal R=\{(\hat y,R\hat y):\hat y\in A\}$ and $\rho(R\hat y,\hat y)=h\circ R(\hat y)/h(\hat y)$ for $\hat y\in A$.
Inserting this into the acceptance probability \eqref{eq:alpha_gen}, which guarantees reversibility of the Metropolis--Hastings kernel $\hat R_{\MH}$, yields
\[ \alpha(\hat y,R\hat y)\ =\
\begin{cases}
\min\left(\frac{\hat\mu_a\circ R(\hat y)}{\hat\mu_a(\hat y)},1\right) & \text{if $\hat\mu_a(\hat y)>0$ and $\hat y_0,\hat y_N\in S^\comp$,} \\
0 & \text{otherwise.}
\end{cases} \]
Redefining the acceptance probability to be one if $\hat\mu_a(\hat y)=0$ and $\hat y_0,\hat y_N\in S^\comp$, which is a $\hat\mu$-null set and hence does not influence reversibility with respect to $\hat\mu$, results in the acceptance probability \eqref{eq:alphaprop}.
\end{proof}

\section{Mixing of Metropolis-adjusted Diffusion Path Monte Carlo}\label{sec:MDPMCmix}

After presenting Metropolis-adjusted Diffusion Path Monte Carlo in the previous section, we now proceed by studying its mixing.
Therefore, we apply the framework to local mixing via Dobrushin contraction developed in \Cref{sec:locDobmix}, which we already used in the analysis of ideal Diffusion Path Monte Carlo and the Proximal Sampler in \Cref{sec:DPMCmix}, and which seamlessly generalizes to the Metropolis-adjusted method.
This yields the, to the best of our knowledge, \emph{first convergence analysis} of Metropolis-adjusted Diffusion Path Monte Carlo.

\begin{theorem}\label{thm:MDPMCmix}
Let $\Pi_{(t)}$ be the transition kernel of Metropolis-adjusted Diffusion Path Monte Carlo with step size $h<2$, noising time $t\in h\mathbb N$, and invariant distribution $\mu\in\calP(\R^d)$, and let $D\subseteq\R^d$.
Then, it holds for all initial distributions $\nu\in\calP(\R^d)$ and $n\geq0$ that
\begin{equation}\label{eq:MDPMCmix}
    \TV\bigl(\nu\Pi_{(t)}^n,\mu\bigr)\ \leq\ \left(\frac{\diam(D)}{2\sqrt{e^t-1}}+2\,\mathfrak P_r(D)\right)^n+\ \P_{\nu}(T<n)\ +\ \P_{\mu}(T<n),
\end{equation}
where
\[ \mathfrak P_r(D) \ :=\ \sup_{x\in D}\int\bigl(1-\alpha(\hat y,R\hat y)\bigr)Q(x,d\hat y) \]
denotes the largest Metropolis rejection probability in $D$, and $T$ the first exit time of Metropolis-adjusted Diffusion Path Monte Carlo from $D$.
\end{theorem}

Under suitable control over rejection and exit probabilities, the theorem yields local geometric mixing for sufficiently large noising times.

\begin{corollary}\label{cor:MDPMCmix}
Consider Metropolis-adjusted Diffusion Path Monte Carlo with step size $h<2$, noising time $t\in h\mathbb N$, and invariant distribution $\mu\in\calP(\R^d)$.
Let $\eps>0$ and $\nu\in\calP(\R^d)$.
If there exists $D\subseteq\R^d$ such that
\[ \mathfrak P_r(D)\ <\ \frac12\quad\text{and}\quad\P_{\eta}\Bigl(T<2\bigl(1-2\,\mathfrak P_r(D)\bigr)^{-1}\log(3\eps^{-1})\Bigr)\ \leq\ \frac\eps3\quad\text{for $\eta\in\{\nu,\mu\}$,} \]
where $\mathfrak P_r(D)$ and $T$ are as in \Cref{thm:MDPMCmix}, and the noising time satisfies
\[ t\ \geq\ \log\Bigl(\bigl(1-2\,\mathfrak P_r(D)\bigr)^{-2}\diam(D)^2+1\Bigr), \]
then the mixing time to accuracy $\eps$ from start in $\nu$ is
\[ \tmix(\eps,\nu)\ \leq\ \Bigl\lceil2\bigl(1-2\,\mathfrak P_r(D)\bigr)^{-1}\log(3\eps^{-1})\Bigr\rceil. \]
\end{corollary}

\Cref{thm:MDPMCmix,cor:MDPMCmix} provide a seamless generalization of the Dobrushin contraction based local mixing analysis from ideal to Metropolis-adjusted Diffusion Path Monte Carlo.
Notably, the results guarantee that, once a domain $D$ in which the dynamics are suitably stable is fixed, mixing of the Metropolis-adjusted method proceeds similarly to the ideal one, if the Metropolis rejection probabilities are controlled uniformly in $D$.
As uncontrolled rejection probabilities along typical trajectories may yield mixing bottlenecks, this assertion is natural.

It is worth emphasizing the scope of the results, admitting \emph{any target and initial probability distributions}.

\begin{remark}[Weaker Metropolis Assumption]\label{rem:weakerMet}
The assumption that rejection probabilities are locally controlled below $1/2$, instead of away from $1$, is not necessary.
Rather, it is a convenient assumption making the ensuing statements especially simple.
Specifically, to establish Dobrushin contraction of $\Pi_{(t)}$, the proof shows that for all $x,y\in D$,
\[ \TV\bigl(\Pi_{(t)}(x,\cdot),\Pi_{(t)}(y,\cdot)\bigr)\ \leq\ \TV\bigl(Q_{(t)}(x,\cdot),Q_{(t)}(y,\cdot)\bigr)\ +\ 2\,\mathfrak P_r(D), \]
based on which contraction requires $\mathfrak P_r(D)<1/2$.
Notably, the previous display separates the contraction of $Q_{(t)}$ from the effect of Metropolis adjustment, leading to a particularly clean statement.

In fact, however, the proof below establishes the stronger estimate
\[ \TV\bigl(\Pi_{(t)}(x,\cdot),\Pi_{(t)}(y,\cdot)\bigr) \leq \TV\bigl(Q_{(t)}(x,\cdot),Q_{(t)}(y,\cdot)\bigr) + \E\Bigl(1-\min\bigl(\alpha(\hat X,R\hat X),\alpha(\hat Y,R\hat Y)\bigr)\Bigr), \]
where $\bigl(\hat X,\hat Y\bigr)$ is a coupling of $Q(x,\cdot)$ and $Q(y,\cdot)$ such that $\bigl(\hat X_N,\hat Y_N\bigr)$ maximally couples $Q_{(t)}(x,\cdot)$ and $Q_{(t)}(y,\cdot)$.
The second term on the right hand side represents the counter probability of mutual acceptance and is bounded above by $2\,\mathfrak P_r(D)$.
In particular, the theorem's assertion \eqref{eq:MDPMCmix} remains to hold with this term replacing $2\,\mathfrak P_r(D)$.
Consequently,
\[ \E\Bigl(1-\min\bigl(\alpha(\hat X,R\hat X),\alpha(\hat Y,R\hat Y)\bigr)\Bigr)\ <\ 1 \]
uniformly for all $x,y\in D$ suffices for the required local Dobrushin contraction of $\Pi_{(t)}$, relaxing the assumption $2\,\mathfrak P_r(D)<1$ made in the theorem.
\end{remark}

\begin{remark}[Explicit Discretization and Score Approximation Errors]
The discretization and score approximation errors enter our results implicitly through Metropolis rejection probabilities.
While the explicit study of these errors lies beyond the scope of this work, there is a growing body of literature on their control, especially in the context of diffusion models, see \cite{Sinho23,Jianfeng23,Benton24,sifanDP}.
Once suitable error bounds are established, resulting in bounds on rejection probabilities, they can be inserted into \Cref{thm:MDPMCmix,cor:MDPMCmix} to make the dependencies on discretization and score approximation errors explicit.
\end{remark}

\begin{remark}[Practical Use: Warm-up Steps]
Our considerations above, specifically the ability of ideal Diffusion Path Monte Carlo to generate a warm start in one transition, see \Cref{sec:gapmix}, motivates the following suggestion for the practical use of the Metropolis-adjusted method:
Before proceeding with the Metropolis-adjusted transitions $\Pi_{(t)}$, one may consider implementing one or a couple of transitions $Q_{(t)}$ omitting adjustment.
Assuming $Q_{(t)}$ to describe the target reasonably well, this produces an informed initialization for $\Pi_{(t)}$.
In particular, it prevents initialization in the tails of the distribution, to which Metropolis adjustment is known to be sensitive, see \cite{Diss}.
\end{remark}

\begin{remark}[Mixing in One Step]
In \Cref{cor:oneDPMC}, we saw that ideal Diffusion Path Monte Carlo, with suitably large noising time, mixes in a single transition to a given accuracy.
This cannot be expected for the Metropolis-adjusted method as the chain remains in its previous state in case of rejection and hence does not proceed to an approximate draw from the target distribution, regardless of noising time.
In particular, for mixing to accuracy $\eps$ in one transition, the rejection probability would need to be of order $\eps$.
On the one hand, in Metropolis-adjusted MCMC methods involving discretization similar to the Euler--Maruyama scheme used here, such as the Metropolis-adjusted Langevin Algorithm (MALA) and Hamiltonian Monte Carlo (HMC), this typically leads to step sizes depending algebraically on the accuracy, while rejection probabilities being controlled above by a constant $c<1$ commonly allows much larger step sizes scaling inverse logarithmically with $\eps^{-1}$, see \cite{Diss}.
On the other hand, this disregards the score approximation error, which may exclude such small rejection probabilities, regardless of discretization step size $h$.
\end{remark}

\begin{proof}[Proof of \Cref{thm:MDPMCmix}]
Let $D\subseteq\R^d$, $\nu\in\calP(\R^d)$, and $n\geq0$.
The proof is based on \Cref{cor:mixlocDob}, which asserts that
\begin{equation*}
    \TV\bigl(\nu\Pi_{(t)}^n,\mu\bigr)\ \leq\ \left(\sup_{x,y\in D}\TV\bigl(\Pi_{(t)}(x,\cdot),\Pi_{(t)}(y,\cdot)\bigr)\right)^n+\ \P_{\nu}(T<n)\ +\ \P_{\mu}(T<n).
\end{equation*}
To conclude, it thus suffices to show
\begin{equation}\label{eq:MetDob}
\sup_{x,y\in D}\TV\bigl(\Pi_{(t)}(x,\cdot),\Pi_{(t)}(y,\cdot)\bigr)\ \leq\ \frac{\diam(D)}{2\sqrt{e^t-1}}\ +\ 2\,\sup_{x\in D}\int\bigl(1-\alpha(\hat y,R\hat y)\bigr)Q(x,d\hat y).
\end{equation}

Therefore, let $x,y\in D$ with $x\ne y$.
Let $\bigl(\hat X,\hat Y\bigr)$ be a coupling of $Q(x,\cdot)$ and $Q(y,\cdot)$, see \eqref{eq:Q}, such that $\bigl(\hat X_N,\hat Y_N\bigr)$ is a maximal coupling of $Q_{(t)}(x,\cdot)$ and $Q_{(t)}(y,\cdot)$, see \eqref{eq:Qft}, that is,
\begin{equation}\label{eq:maximalNhalf}
    \mathbb P\bigl(\hat X_N\ne\hat Y_N\bigr)\ =\ \TV\bigl(Q_{(t)}(x,\cdot),Q_{(t)}(y,\cdot)\bigr).
\end{equation}
The random variables
\begin{align*}
X\ =\ \hat X_N\,\ind_{\{U\leq\alpha(\hat X,R\hat X)\}} + x\,\ind_{\{U>\alpha(\hat X,R\hat X)\}},\quad
Y\ =\ \hat Y_N\,\ind_{\{U\leq\alpha(\hat Y,R\hat Y)\}} + y\,\ind_{\{U>\alpha(\hat Y,R\hat Y)\}}
\end{align*}
with $U\sim\Unif(0,1)$ independent of $(\hat X,\hat Y)$, define a coupling $(X,Y)$ of the Metropolis-adjusted Diffusion Path Monte Carlo transitions $\Pi_{(t)}(x,\cdot)$ and $\Pi_{(t)}(y,\cdot)$.
Hence, by the coupling characterization of total variation, we have
\[ \TV\bigl(\Pi_{(t)}(x,\cdot),\Pi_{(t)}(y,\cdot)\bigr)\ \leq\ \P(X\ne Y). \]
Note that, if both chains are accepted, $X\ne Y$ holds if and only if $\hat X_N\ne\hat Y_N$.
Thus,
\[ \P(X\ne Y)
\ \leq\ \P\bigl(\hat X_N\ne\hat Y_N\bigr)\ +\ \P\Bigl(U>\min\bigl(\alpha(\hat X,R\hat X),\alpha(\hat Y,R\hat Y)\bigr)\Bigr). \]
Inserting \eqref{eq:maximalNhalf} shows
\begin{equation*}
\begin{aligned}[t]
\TV\bigl(\Pi_{(t)}(x,\cdot),\Pi_{(t)}(y,\cdot)\bigr)\ \leq\ \, &\TV\bigl(Q_{(t)}(x,\cdot),Q_{(t)}(y,\cdot)\bigr) \\
&+\ \P\Bigl(U>\min\bigl(\alpha(\hat X,R\hat X),\alpha(\hat Y,R\hat Y)\bigr)\Bigr).
\end{aligned}
\end{equation*}

As $Q_{(t)}=\Qf_{(t)}\Qb_{(t)}$ and total variation is non-increasing under the action of any Markov kernel,
\[ \TV\bigl(Q_{(t)}(x,\cdot),Q_{(t)}(y,\cdot)\bigr)\ \leq\ \TV\bigl(\Qf_{(t)}(x,\cdot),\Qf_{(t)}(y,\cdot)\bigr). \]
Further, since 
\[ \Qf_{(t)}(x,\cdot)\ =\ \N\left((1-\tfrac12h)^{t/h}x,\frac{1-(1-\frac12h)^{2t/h}}{1-\frac14h}I_d\right)\quad\text{for all $x\in\R^d$}, \]
see \eqref{eq:Qf} and \eqref{eq:Qft}, using \cite[Thm.~1]{Barsov86} similarly as in \eqref{eq:TV_Pf} above, the total variation on the right hand side evaluates to
\begin{align*}
\TV\bigl(\Qf_{(t)}(x,\cdot),\Qf_{(t)}(y,\cdot)\bigr)
\ &=\ 2\,\Phi\Biggl(\frac12\left(\frac{(1-\frac12h)^{-2t/h}-1}{1-\frac14h}\right)^{-1/2}|x-y|\Biggl)\ -\ 1 \\
&\leq\ \frac{1}{\sqrt{2\pi}}\left(\frac{(1-\frac12h)^{-2t/h}-1}{1-\frac14h}\right)^{-1/2}|x-y| \\
&\leq\ \frac{|x-y|}{2\sqrt{e^t-1}},
\end{align*}
where we used that $1-\frac12h\leq e^{-\frac12 h}$ for $0<h\leq2$ in the last step.
This shows, for all $x,y\in D$,
\begin{equation}\label{eq:MetDobstr}
    \TV\bigl(\Pi_{(t)}(x,\cdot),\Pi_{(t)}(y,\cdot)\bigr)\ \leq\ \frac{\diam(D)}{2\sqrt{e^t-1}}\ +\ \P\Bigl(U>\min\bigl(\alpha(\hat X,R\hat X),\alpha(\hat Y,R\hat Y)\bigr)\Bigr).
\end{equation}
Finally, simplifying the second term on the right hand side along
\begin{align*}
\P\Bigl(U>\min\bigl(\alpha(\hat X,R\hat X),\alpha(\hat Y,R\hat Y)\bigr)\Bigr)
\ &=\ \E\Bigl(1-\min\bigl(\alpha(\hat X,R\hat X),\alpha(\hat Y,R\hat Y)\bigr)\Bigr) \\
&\leq\ \E\bigl(1-\alpha(\hat X,R\hat X)\bigr)+\E\bigl(1-\alpha(\hat Y,R\hat Y)\bigr) \\
&\leq\ 2\,\sup_{x\in D}\int\bigl(1-\alpha(\hat y,R\hat y)\bigr)Q(x,d\hat y)
\end{align*}
establishes \eqref{eq:MetDob}, concluding the proof.
\end{proof}

\begin{remark}
The proof, specifically of \eqref{eq:MetDob}, can be simplified via the triangle inequality for total variation.
Namely, it holds
\begin{align*}
\TV\bigl(\Pi_{(t)}(x,\cdot),\Pi_{(t)}(y,\cdot)\bigr)
\ &\leq\ 
\begin{aligned}[t]
&\TV\bigl(\Pi_{(t)}(x,\cdot),Q_{(t)}(x,\cdot)\bigr) + \TV\bigl(Q_{(t)}(x,\cdot),Q_{(t)}(y,\cdot)\bigr) \\
&+ \TV\bigl(Q_{(t)}(y,\cdot),\Pi_{(t)}(y,\cdot)\bigr)
\end{aligned} \\
&\leq\ \frac{\diam(D)}{2\sqrt{e^t-1}}\ +\ 2\,\sup_{x\in D}\int\bigl(1-\alpha(\hat y,R\hat y)\bigr)Q(x,d\hat y),
\end{align*}
where the second step uses the bound on the total variation distance between $Q_{(t)}(x,\cdot)$ and $Q_{(t)}(y,\cdot)$ established above, and the fact that, in the language of the above proof, $\bigl(X,\hat X_N\bigr)$ couples $\Pi_{(t)}(x,\cdot)$ and $Q_{(t)}(x,\cdot)$ so that
\[ \TV\bigl(\Pi_{(t)}(x,\cdot),Q_{(t)}(x,\cdot)\bigr)\ \leq\ \P\bigl(X\ne\hat X_N\bigr)\ \leq\ \int\bigl(1-\alpha(\hat y,R\hat y)\bigr)Q(x,d\hat y). \]
However, this argument does not establish the stronger estimate \eqref{eq:MetDobstr} involving the counter probability of mutual acceptance, which admits weaker assumptions on the Metropolis correction than $\mathfrak P_r(D)<1/2$, see \Cref{rem:weakerMet}.
\end{remark}

To close our considerations, we show that the Metropolis acceptance probability in high-dimensional standard Gaussian targets, for certain step sizes, is controlled away from zero in typical stationary draws while degenerating exponentially in the origin.
Therefore, at such step sizes, Metropolis-adjusted Diffusion Path Monte Carlo locally mixes rapidly in a high-probability region, contrasted with much slower global mixing due to severe traps in the low-probability region around the origin, similar to the seldom-trapped chain discussed in the introduction to motivate local geometric mixing.
This illustrates the need to localize the analysis to accurately capture local geometric mixing, either through the approach presented here or its alternatives discussed in the introduction.

\begin{example}[Metropolis-adjusted Diffusion Path Monte Carlo as Seldom-trapped Chain]
Consider Metropolis-adjusted Diffusion Path Monte Carlo with invariant distribution $\mu=\N(0,I_d)$, step size $h\in(0,2)$, noising time $t\in h\mathbb N$, using the exact score $\mathfrak s(x)=-x$.
Write $N=2t/h$, as well as $\mathfrak P_r(x):=\mathfrak P_r(\{x\})$ and $\mathfrak P_a(x):=1-\mathfrak P_r(x)$ for the rejection and acceptance probabilities in $x\in\R^d$.
Below, we show that, in typical draws $x\sim\mu$, acceptance probabilities are controlled away from zero with high probability for suitable $h=\mathcal O(d^{-1/2})$, whereas in the origin,
\[ \mathfrak P_a(0)\ \leq\ \exp\left(-\frac{1-e^{-2t}}8hd\right), \]
that is, acceptance probabilities degenerate to zero exponentially as $d\to\infty$ unless $h=\mathcal O(d^{-1})$.
In particular, for suitable $h=\Theta(d^{-1/2})$, acceptance probabilities are controlled away from zero with high probability in typical stationary draws while degenerating in the origin.

To see the claimed bounds, note that, since noising and denoising steps use the same transitions, the Metropolis acceptance ratio telescopes, leading to
\begin{equation}\label{eq:Gaussexalpha}
    \alpha(\hat y,R\hat y)\ =\ \exp\left(-\frac h8\bigl(|\hat y_N|^2-|\hat y_0|^2\bigr)^+\right),
\end{equation}
where $\cdot^+:=\max(\cdot,0)$.

On the one hand, using that under $Q(x,\cdot)$,
\begin{equation}\label{eq:GaussexN}
    \hat y_N\ \sim\ \N\left(\bigl(1-\tfrac12h\bigr)^Nx,\frac{1-(1-\frac12h)^{2N}}{1-\frac14h}I_d\right),
\end{equation}
it holds
\begin{align*}
    \mathfrak P_r(x)\ &=\ \int\bigl(1-\alpha(\hat y,R\hat y)\bigr)\,Q(x,d\hat y)\ \leq\ \frac h8\int\bigl||\hat y_N|^2-|x|^2\bigr|\,Q(x,d\hat y)\\
    &=\ \frac h8\bigl(1-(1-\tfrac12h)^{2N}\bigr)\bigl((1-\tfrac14h)^{-1}d-|x|^2\bigr)+\mathcal O\bigl(h|x|\bigr).
\end{align*}
Inserting that $|x|^2=d+\mathcal O(d^{1/2})$ with high probability for $x\sim\N(0,I_d)$, see \cite{Ledoux01}, shows that the rejection probability in typical draws from the target are
\[ \mathfrak P_r(x)\ =\ \mathcal O\bigl(hd^{1/2}+h^2d\bigr)\quad\text{with high probability,} \]
ensuring acceptance probabilities to be controlled away from zero for suitable step sizes $h=\mathcal O(d^{-1/2})$.

On the other hand, using \eqref{eq:Gaussexalpha} and \eqref{eq:GaussexN} yields
\begin{align*} \mathfrak P_a(0)\ &=\ \int\alpha(\hat y,R\hat y)\,Q(0,d\hat y)\ \leq\ \int\exp\left(-\frac h8|\hat y_N|^2\right)\,Q(0,d\hat y)\\
&=\ \left(\frac{4-h}{4-h(1-\frac12h)^{2N}}\right)^{d/2}\ \leq\ \left(1-\frac{1-e^{-2t}}4h\right)^{d/2}\\
&\leq\ \exp\left(-\frac{1-e^{-2t}}8hd\right).
\end{align*}
\end{example}

\section*{Acknowledgements}

The author gratefully acknowledges funding by the Deutsche Forschungsgemeinschaft (DFG, German Research Foundation) – CRC 1720 – 539309657

\printbibliography

\end{document}